\documentclass{amsart}[12pt]
  \usepackage{tabularx}
  \usepackage{enumerate}
  \usepackage{mathrsfs}
\usepackage[all]{xy}

\usepackage{amsmath}
\usepackage{amsfonts}
\usepackage{amssymb}
\usepackage{amsthm}

\usepackage{url}

\usepackage{xspace}  

\newcommand{\dr}{\textup{dR}}
\newcommand{\hdr}{H_{\dr}}
\newcommand{\ndiv}{\nmid}
\newcommand{\Ga}{\mathbb{G}_a}

\newcommand{\lag}{\Psi}
\DeclareMathOperator{\Lie}{Lie}

\DeclareMathOperator{\chr}{char}

\DeclareMathOperator{\Ann}{Ann}

\DeclareMathOperator{\Pic}{Pic}

\DeclareMathOperator{\Reg}{Reg}
\DeclareMathOperator{\Res}{Res}

\DeclareFontEncoding{OT2}{}{} 
  \newcommand{\textcyr}[1]{%
    {\fontencoding{OT2}\fontfamily{wncyr}\fontseries{m}\fontshape{n}%
     \selectfont #1}}
\newcommand{\Sha}{{\mbox{\textcyr{Sh}}}}

\newcommand{\cL}{\mathcal{L}}

\newcommand{\cO}{\mathcal{O}}

\DeclareMathOperator{\Spec}{Spec}
\DeclareMathOperator{\Div}{Div}

\DeclareMathOperator{\res}{res}

\DeclareMathOperator{\sign}{sign}

\newcommand{\into}{\rightarrow}

\DeclareMathOperator{\tors}{tors}

\newcommand{\ra}{\rightarrow}

\newcommand{\lra}{\longrightarrow}

\newcommand{\comment}[1]{}
\newcommand{\Q}{\mathbb{Q}}

\newcommand{\R}{\mathbb{R}}

\newcommand{\C}{\mathbb{C}}

\newcommand{\T}{\mathbb{T}}

\newcommand{\Z}{\mathbb{Z}}

\newcommand{\F}{\mathbb{F}}

\renewcommand{\O}{\mathcal{O}}

\DeclareMathOperator{\disc}{disc}
\DeclareMathOperator{\ord}{ord}
\DeclareMathOperator{\GL}{GL}

\DeclareMathOperator{\Gal}{Gal}

\newcommand{\cB}{\mathcal{B}}

\DeclareMathOperator{\Hom}{Hom}

\theoremstyle{plain}
\newtheorem{theorem}{Theorem}[section]
\newtheorem{proposition}[theorem]{Proposition}
\newtheorem{corollary}[theorem]{Corollary}

\newtheorem{lemma}[theorem]{Lemma}

\newtheorem{conjecture}[theorem]{Conjecture}

\theoremstyle{definition}
\newtheorem{definition}[theorem]{Definition}

\newtheorem{remark}[theorem]{Remark}

\numberwithin{equation}{section}
\numberwithin{figure}{section}
\numberwithin{table}{section}

\newcounter{listnum}

\newtheorem{algorithm}[theorem]{Algorithm}
{\begin{enumerate}\setlength{\itemsep}{0.1ex}}{\end{enumerate}}

\usepackage{color}
\usepackage{listings} 
\lstdefinelanguage{Sage}[]{Python}
{morekeywords={True,False,sage,singular},
sensitive=true}
\definecolor{lightyellow}{rgb}{1,1,.86}
\definecolor{dblackcolor}{rgb}{0.0,0.0,0.0}
\definecolor{dbluecolor}{rgb}{.01,.02,0.7}
\definecolor{dredcolor}{rgb}{0.8,0,0}
\definecolor{dgraycolor}{rgb}{0.30,0.3,0.30}
\definecolor{graycolor}{rgb}{0.35,0.35,0.35}

  \title[$p$-adic BSD for modular abelian varieties]{A $p$-adic analogue of the conjecture of Birch and
        Swinnerton-Dyer for modular abelian varieties}

 \author{Jennifer S. Balakrishnan}
\address{Jennifer S. Balakrishnan, Mathematical Institute, University of Oxford, Woodstock Road, Oxford OX2 6GG, UK}
\email{balakrishnan@maths.ox.ac.uk}

\author{J. Steffen M\"{u}ller}
\address{J. Steffen M\"{u}ller, Institut f\"ur Mathematik, Carl von Ossietzky
Universit\"{a}t Oldenburg, 26111 Oldenburg, Germany}
\email{jan.steffen.mueller@uni-oldenburg.de}

\author{William A. Stein}
\address{William A. Stein, Department of Mathematics, University of Washington, Seattle, Box 354350
    WA 98195, USA }
\email{wstein@uw.edu}
\date{\today}

\begin{document}
  \maketitle

  \begin{abstract}  Mazur, Tate, and Teitelbaum gave a $p$-adic analogue of the Birch and
Swinnerton-Dyer conjecture for elliptic curves. We provide a generalization of their
conjecture in the good ordinary case to higher dimensional modular abelian varieties over
    the rationals by constructing the $p$-adic $L$-function of a modular abelian variety and showing that it satisfies the appropriate interpolation property. This relies on a careful normalization of the $p$-adic $L$-function, which we achieve by a comparison of periods.  Our generalization agrees with the conjecture of Mazur, Tate, Teitelbaum in dimension 1 and the classical Birch Swinnerton-Dyer conjecture formulated by Tate in rank 0.  We describe the theoretical techniques used to formulate the conjecture and give numerical evidence supporting the conjecture in the case when the modular abelian variety is of dimension 2.   \end{abstract}

      \section{Introduction}
The Birch and Swinnerton-Dyer (BSD) conjecture gives a precise relationship between several arithmetic invariants of an abelian variety $A$ over a number field $K$. As formulated by Tate \cite{tate:bsd}, the conjecture states the following:
\begin{conjecture}[BSD conjecture for abelian varieties]\label{bsdabvar} Let $A$ be an
    abelian variety of dimension $g$ over a number field $K$, and let $A^{\vee}$ be its dual. Then the Mordell-Weil rank $r$ of $A(K)$ is equal to the
    analytic rank $\ord_{s=1}L(A,s)$ of $A$ and  $$\lim_{s \ra 1}(s-1)^{-r}L(A,s) =
    \frac{\Omega_A \cdot
    |\Sha(A/K)| \cdot \Reg(A/K)  \cdot \prod_{v} c_v}{\sqrt{|D_K|}^g\cdot|A(K)_{\tors}|\cdot|A^{\vee}(K)_{\tors}|},$$
    where $D_K$ is the absolute discriminant of $K$, $\Omega_A$ is the real period,
    $\Reg(A/K)$ is the regulator,  $c_v$ is the
    Tamagawa number at a finite place $v$ of $K$, $\Sha(A/K)$ the Shafarevich-Tate
group of $A$ and $A(K)_{\tors}$ is the torsion subgroup of $A(K)$.\end{conjecture}

Note that this conjecture relies on two assumptions: that the Shafarevich-Tate group
$\Sha$ is finite and that the $L$-series can be analytically continued to $s=1$. An
analytic continuation is known to exist for modular abelian varieties over $\Q$, where an
abelian variety is said to be {\em modular} if it is a quotient of $J_1(N)$ for some level $N$.
In particular, for an elliptic curve $E/\Q$ of rank $r$, the BSD conjecture predicts
\begin{conjecture}[BSD conjecture for elliptic curves]\label{bsdec}Let $E$ be an elliptic
curve over $\Q$. Then the Mordell-Weil rank $r$ of $E(\Q)$ is equal to the
    analytic rank of $E$ and $$\lim_{s \ra 1}(s-1)^{-r}L(E,s)  = \frac{\Omega_E \cdot
|\Sha(E/\Q)| \cdot \Reg(E/\Q) \cdot \prod_{p} c_p}{|E(\Q)_{\tors}|^2}.$$\end{conjecture}

In 1986, Mazur, Tate, and Teitelbaum \cite{mtt} gave a $p$-adic analogue of this
conjecture for an elliptic curve $E$ over the rationals and a prime $p$ of good ordinary
or multiplicative reduction.  Much work has been done towards a proof of the conjecture,
and more is known about the $p$-adic conjecture than its classical counterpart.  We give a
brief overview of the circle of ideas involved; see also the recent work of Stein and Wuthrich \cite{stein-wuthrich:shark}.
For simplicity, we assume that $p>2$.

Let $_{\infty}G$ denote the Galois group $\Gal(\Q(\mu_{p^{\infty}})/\Q)$, where
$\Q(\mu_{p^{\infty}})$ is the cyclotomic extension of $\Q$ obtained by adjoining all
$p$-power roots of unity, let $\kappa:{}_{\infty}G \ra \Z_p^{\times}$ denote the cyclotomic character and let $\gamma$
be a topological generator of $\Gamma =\, _{\infty}G^{(p-1)}$.
For an elliptic curve $E/\Q$ and a prime $p$ such that $E$ has good or multiplicative
reduction at $p$, we denote the $p$-adic regulator, divided by $\log_p(\kappa(\gamma))^r$, by
    $\Reg_\gamma(E/\Q)$ and we let $\mathcal{L}_p(E,T)$
denote the series expansion of the $p$-adic $L$-function $L_p(E,s)$ associated to $E$ in $T = \kappa(\gamma)^{s-1} - 1$.
\begin{conjecture}[$p$-adic BSD conjecture for elliptic curves]\label{mtt}
    Let $E$ be an elliptic curve over $\Q$ and let $p$ be a prime number such that $E$ has
good ordinary or multiplicative reduction at $p$.
    \begin{enumerate}[(i)]
    \item The order of vanishing $\ord_T(\mathcal{L}_p(E,T))$  of  $\mathcal{L}_p(E,T)$ at
$T=0$ is equal to the rank $r$ of $E(\Q)$ if $E$ has good ordinary or nonsplit multiplicative reduction at $p$. If $E$ has split multiplicative reduction at $p$, then $\ord_T(\mathcal{L}_p(E,T)) = r+1$.
\item If $E$ has good ordinary or nonsplit multiplicative reduction at $p$, then the leading term $\mathcal{L}_p^*(E,0)$ satisfies
    \begin{equation}\label{bsdp}\mathcal{L}_p^*(E,0) = \epsilon_p(E)\cdot  \frac{|\Sha(E/\Q)|
    \cdot \Reg_{\gamma}(E/\Q) \cdot \prod_{v} c_{v}}{|E(\Q)_{\tors}|^2},\end{equation}
    where $\epsilon_p(E)=(1-\alpha^{-1})^b$ for a unit root $\alpha$ of $x^2 - a_p x + p \in \Q_p[x]$ (with $a_p$
    the Hecke eigenvalue of the newform associated to $E$) and $b$ is~2 if $E$ has good ordinary reduction at $p$ and~1 if $E$ has
nonsplit multiplicative reduction at $p$.

    If $E$ has split multiplicative reduction at $p$, then
        \begin{equation*}\mathcal{L}_p^*(E,0)= \frac{\mathscr{S}_p}{\log_p(\kappa(\gamma))} \cdot \frac{|\Sha(E/\Q)| \cdot
    \Reg_{\gamma}(E/\Q) \cdot \prod_{v} c_{v}}{|E(\Q)_{\tors}|^2},\end{equation*}
    where $\mathscr{S}_p=\frac{\log_p(q_E)}{\ord_p(q_E)}$
    and $q_E$ is the Tate period of $E$ over $\Q_p$.
\end{enumerate}
\end{conjecture}
For primes of supersingular reduction, the $p$-adic $L$-function and the $p$-adic regulator can also be defined,  and a $p$-adic BSD conjecture has been formulated by Bernardi and Perrin-Riou
\cite{bernardi-perrin-riou}.

Much work has been done toward a proof of Conjecture~\ref{mtt}, but since most of it 
also applies to higher-dimensional modular abelian varieties,
we defer a discussion of the known results to the end of this introduction.

We note that in the case of elliptic curves, the classical BSD conjecture (Conjecture~\ref{bsdec}) shares many of the same
arithmetic quantities with the $p$-adic BSD conjecture (Conjecture~\ref{mtt}); the main difference is that the regulator
and $L$-series are replaced with $p$-adic analogues. In particular, the conjectures
are equivalent if the rank is~0 and $p$ has good ordinary or nonsplit multiplicative reduction.
Consequently, one might expect that a statement like Conjecture~\ref{bsdabvar} could be
formulated and studied for a modular abelian variety $A/\Q$ associated to a newform $f$; this is the goal of the
present paper.

The main theoretical difficulty in formulating the conjecture was that, prior to the
present work, no construction of \emph{the} $p$-adic $L$-function associated to a modular
abelian variety $A/\Q$ of dimension greater than~1 seemed to be known. Nevertheless, one knows that by the general motivic framework outlined by Coates
\cite{coates:bourbaki}, a $p$-adic $L$-series attached to $A$ should interpolate the complex $L$-series $L(A,s)$
at special values, so it seems plausible that it could be defined, similar to $L(A,s)$, as the product of
$p$-adic $L$-functions $L_p(f^\sigma,s)$ associated to the Galois conjugates $f^\sigma$ of $f$.

However, there is no obvious canonical choice for $L_p(f^\sigma,s)$, since picking a $p$-adic
$L$-function associated to $f^\sigma$ requires picking a Shimura period (see Theorem~\ref{thm:shimura})
for each $f^\sigma$, that is, a complex number $\Omega^+_{f^\sigma}$ such that $L(f^\sigma,1)/\Omega^+_{f^\sigma}$ is
algebraic.
In the case of elliptic curves, this is not an issue, since we can choose $\Omega^+_f$ to be
the real period $\Omega^+_E$ of the associated elliptic curve $E$.

We show that there is a natural way to pick a set of Shimura periods, which allows us to construct the $p$-adic $L$-function attached to $A/\Q$. Indeed, an extension of Conjecture~\ref{mtt} to modular abelian varieties  should be equivalent to Conjecture~\ref{bsdabvar} in rank~0.
Since the latter involves the real period $\Omega^+_A$ associated to $A$, this equivalence
forces the  product of the Shimura periods $\Omega^+_{f^\sigma}$ to equal $\Omega^+_A$.
We show that we can choose a set of Shimura periods with this property (see
Theorem~\ref{thm:omegas}), which allows us to normalize our $p$-adic $L$-function.

In this way we are able to explicitly construct, essentially generalizing the treatment in
\cite{mtt}, a $p$-adic $L$-function associated to $A$ with the expected interpolation
property, in the case where $p$ is a prime of good ordinary reduction (see \eqref{Lpf} and
\eqref{cLpf}).
We keep the notation introduced prior to Conjecture~\ref{mtt} and extend it to the case of
modular abelian varieties.
We also define the $p$-adic multiplier $\epsilon_p(A)$ as follows: fix a prime $\wp\mid{}p$ of
the number field $K_f$ generated by the Hecke eigenvalues of $f$ and let $\alpha^\sigma$ denote
the unit root of $x^2-\sigma(a_p)x+p \in (K_f)_\wp[x]$, where $\sigma:K_f\hookrightarrow\C$ is an
    embedding.
For a Galois conjugate $f^\sigma$ of $f$ define
$\epsilon_p(f^\sigma)=(1-1/(\alpha^\sigma))^2$ and
define $\epsilon_p(A)$ to be the product of the $p$-adic multipliers $\epsilon_p(f^\sigma)$ over all distinct
Galois conjugates of $f$.

We make the following $p$-adic BSD conjecture:
\begin{conjecture}\label{pbsd} Let $A/\Q$ be a modular abelian variety associated to a
    newform $f$ and let $p$ be a prime
number such that $A$ has good ordinary reduction at $p$.
Then the Mordell-Weil rank $r$ of $A$ equals
      $\ord_T(\cL_p(A,T))$ and
      \begin{equation}\label{eq-pbsd}\mathcal{L}_p^*(A,0) =\epsilon_p(A)
\cdot\frac{ |\Sha(A/\Q)|
      \cdot \Reg_{\gamma}(A/\Q) \cdot \prod_v  c_v
}{|A(\Q)_{\tors}|\cdot|A^{\vee}(\Q)_{\tors}|},\end{equation} where $\mathcal{L}_p^*(A,0)$
is the leading coefficient of the $p$-adic $L$-series $\mathcal{L}_p(A,T)$. \end{conjecture}

Note that the conjecture of Mazur, Tate, and Teitelbaum in the good ordinary case is a
special case of Conjecture~\ref{pbsd}.
Moreover, if the rank of $A/\Q$ is zero, then our conjecture is equivalent to the
classical BSD conjecture due to the interpolation property
\[
    \mathcal{L}_p(A,0) = L_p(A,1) =   \epsilon_p(A) \cdot \frac{L(A,1)}{\Omega^+_A}.
\]
Most progress toward proving the conjecture of Mazur, Tate, and Teitelbaum uses Iwasawa theory
and many results remain valid in our setup.
Unfortunately, Iwasawa-theoretic results typically only assert equality up to a $p$-adic
unit, whereas our Conjecture~\ref{pbsd} asserts full equality.
We restrict to the good ordinary case from now on, although most results have
supersingular or multiplicative analogues.
See \cite[\S6,7]{stein-wuthrich:shark} for a summary of such analogues in the elliptic
curves case.

Let $X(A/_\infty\Q)$ denote the Pontryagin dual of the $p$-Selmer group of $A$ and let
$\Lambda$ be the completed group algebra $\Z_p[[\Gamma]]$.
Using $p$-adic Hodge theory, Kato \cite{kato:secret} has shown that $X(A/_\infty\Q)$ is a torsion
$\Lambda$-module.
Hence we can associate a characteristic series $f_A(T)\in \Z_p[[T]]$, well-defined up to a factor
in $\Z_p[[T]]^\times$, to $X(A/_\infty\Q)$.

\begin{conjecture}(Main conjecture of Iwasawa theory for abelian varieties with good
    ordinary reduction)
There exists an element $u(T)\in\Lambda^\times$ such that
\[
\cL_p(A,T)=f_A(T)\cdot u(T).
\]
\end{conjecture}
If $A=E$ is an elliptic curve, then the main conjecture is known to be a theorem in many cases.
If $E$ has complex multiplication, then a proof is due to Rubin
\cite{rubin:main-conjectures}.
Many other cases have been proven, culminating in the work \cite{skinner-urban} of Skinner
and Urban.
See \cite[\S7]{stein-wuthrich:shark} for an overview.  The following unconditional result is due to Kato \cite{kato:secret}:
\begin{theorem}\label{thm:kato}(Kato)
Let $A$ be an elliptic curve.
There is an integer $m\ge0$ such that $f_A(T)$ divides $p^m\cL_p(A,T)$.

\end{theorem}
The following result of Perrin-Riou \cite{perrin-riou:elliptiques} and Schneider \cite{schneider:heightII} relates $f_A(T)$ to the right hand side of~\eqref{eq-pbsd}:
\begin{theorem}\label{thm:prs}(Perrin-Riou, Schneider)
    The order of vanishing $\ord_{T=0} f_A(T)$ is greater than or equal to the rank $r$ of
    $A/\Q$.
    Equality holds if and only if the $p$-adic height pairing on $A$ is nondegenerate and the $p$-primary
part $\Sha(A/\Q)(p)$ of the Shafarevich-Tate group of $A$ is finite, in which case the
 leading coefficient of $f_A(T)$ has the same valuation as
\[
\epsilon_p(A)\cdot\frac{ |\Sha(A/\Q)(p)|
      \cdot \Reg_{\gamma}(A/\Q) \cdot \prod_v  c_v
}{|A(\Q)_{\tors}|\cdot|A^{\vee}(\Q)_{\tors}|}.
\]
\end{theorem}
See for instance \cite[Theorem~2']{schneider:heightII}, noting that we have
$\ord_p(\epsilon_p(A))=2\ord_p(N_p)$, where $N_p$ is the number of $\F_p$-rational points on the reduction
of $A$ over $\F_p$.

As a corollary of Theorem~\ref{thm:kato} and Theorem~\ref{thm:prs}, we have that
\begin{equation}\label{rank-ineq}
    \ord_{T=0} \cL_p(A,T)\ge\ord_{T=0} f_A(T)\ge r=\mathrm{rank}(A(\Q)),
\end{equation}
if $A$ is an elliptic curve, so one direction of the first part of
Conjecture~\ref{mtt} (with~\eqref{rank-ineq} suitably modified in the case of multiplicative
reduction) is already known.
Moreover, the second part of our Conjecture~\ref{pbsd} is consistent with the Main Conjecture, since the
latter implies that
the leading coefficients of the $p$-adic $L$-series and the characteristic series have the
same valuation.
If $\cL_p(E,0)\ne0$, then  Conjecture~\ref{mtt}, part (ii) is also known up to a rational
factor  (see \cite{kato:secret}, \cite{perrin-riou:elliptiques}, and also the exposition
in \cite[\S8]{stein-wuthrich:shark}).  The primes appearing in this factor can be
determined explicitly using \cite{skinner-urban}.
In the good ordinary case a similar result also holds when the $p$-adic analytic rank
is~1, under an additional hypothesis; see \cite[\S9]{stein-wuthrich:shark}.
This follows from work of Perrin-Riou \cite{perrin-riou:heegner} and Kato
\cite{kato:secret}.

Historically speaking, numerical evidence played a crucial role in the formulation of Conjecture~\ref{mtt} \cite[\S II.12]{mtt}.
Gathering evidence for Conjecture~\ref{pbsd} would require two computations
independent of the usual Birch and Swinnerton-Dyer conjecture: the computation of $p$-adic
regulators of $A$, as well as  the computation of special values of the $p$-adic
$L$-function attached to $A$. We give algorithms to compute these quantities and provide
the first numerical verification for Conjecture~\ref{pbsd} by considering the modular abelian
varieties of dimension~2 and rank~2 in \cite{empirical} and the Jacobian of a twist of $X_0(31)$ of
rank~4.

Our aim is to give a self-contained discussion of the $p$-adic Birch and Swinnerton-Dyer
conjecture for modular abelian varieties. To that end, we discuss both the theoretical and
the algorithmic aspects of $p$-adic special values and $p$-adic regulators.  This paper is
structured as follows: in $\S 2$, we give a construction of the $p$-adic $L$-series
attached to modular abelian varieties, making explicit certain aspects of \cite{mtt}.
This allows us to compute $p$-adic special values. In $\S 3$, we take a look at the
$p$-adic regulator attached to an abelian variety, focusing on the case when the abelian
variety is the Jacobian of a hyperelliptic curve. We begin by reviewing the work of
Coleman and Gross \cite{coleman-gross}, which gives the $p$-adic height pairing on
Jacobians of curves in terms of local height pairings. We discuss the two types of
local height pairings which arise and give an algorithm to compute $p$-adic
heights, which allows us to compute $p$-adic regulators.  In $\S4$ and $\S 5$ we present  the evidence for the conjecture in dimension 2.

\section*{Acknowledgements}
We thank Robert Pollack, Chris Wuthrich,  Kiran Kedlaya,  Shinichi Kobayashi, Bjorn Poonen, David Holmes and
Michael Stoll for several helpful conversations and David Roe and Robert Pollack for their work implementing an algorithm to compute $p$-adic $L$-series.  The first author was supported by NSF grant DMS-1103831.  The second author was supported by DFG grants STO~299/5-1 and KU~2359/2-1.  The third author was supported by NSF Grants DMS-1161226 and DMS-1147802.

      \section{$p$-adic $L$-functions attached to modular abelian varieties}
In this section, we construct the $p$-adic $L$-function attached to a newform $f\in
S_2(\Gamma_0(N))$, making explicit a few aspects of \cite[\S I.10]{mtt}.
This depends on the choice of a Shimura period as in Theorem~\ref{thm:shimura}.
In order to pin down the $p$-adic $L$-function we want, we relate the Shimura periods of
$f$ and its Galois conjugates to the real period of the abelian variety $A_f$ attached to $f$.
This leads to a definition of a $p$-adic $L$-function for $A_f$ which satisfies the
expected interpolation property~\eqref{interpol}.
Finally, we discuss how this $p$-adic $L$-function can be computed in practice.

 \subsection{Periods} Let $N$ be a positive integer and let $X_0(N)$ be the modular curve of level $N$.
The Jacobian $J_0(N)$ of $X_0(N)$ is an abelian
variety over $\Q$ of dimension equal to the genus of $X_0(N)$, which is equipped with an
action of the Hecke algebra $\T$. The space $S := S_2(\Gamma_0(N))$ of cusp forms of
weight 2 on $\Gamma_0(N)$ is a module over $\T$.
Let $f(z) = \sum_{n=1}^{\infty} a_n e^{2\pi i n z} \in S$ be a newform, let
$K_f$ be the totally real number field $\Q(\ldots, a_n, \ldots)$, and let $I_f$ denote the annihilator $\Ann_{\T}(f)$ of $f$
in $\T$.  Following Shimura \cite{shimura:jac}, we have that the quotient $$A_f =
J_0(N)/I_f J_0(N)$$ is an abelian variety over $\Q$ of dimension $g = [K_f: \Q]$ which is equipped with a faithful action of $\T/I_f$.
Moreover, $A_f$ is an \emph{optimal quotient} of $J_0(N)$ in the sense that the kernel of
$J_0(N) \ra A_f$ is connected.
For ease of notation, we will drop the subscript $f$ and write $A=A_f$.
\begin{remark}
    We assume that $f \in S_2(\Gamma_0(N))$ for convenience, because we need that $K_f$
    is totally real in order
    for Theorem~\ref{thm:shimura} and Theorem~\ref{thm:omegas} to hold precisely as
    stated.
    However, in the case where $K_f$ is a CM-field, Shimura \cite{shimura:periods} has proved that a slightly
    modified version of Theorem~\ref{thm:shimura} continues to hold.
    Using this, one can prove a result that is analogous to
    Theorem~\ref{thm:omegas} for arbitrary newforms $f\in S_2(\Gamma_1(N))$.
\end{remark}
There is a complex-valued pairing $\langle\;,\;\rangle$ on $S\times H_1(X_0(N),\Z)$, given by
integration:
\[
\langle h,\gamma\rangle=2\pi i\int_\gamma h(z)dz.
\]
This pairing induces a natural $\T$-module homomorphism
\[
\Phi:H_1(X_0(N),\Z)\to\Hom_\C(S,\C),
\]
called the {\em period mapping}.

Let $G_f$ be the set of embeddings $\sigma:K_f\into\C$ .
If $\sigma\in G_f$, we let $f^\sigma$ denote the conjugate of $f$ by $\sigma$.
We denote the complex vector space generated by the Galois conjugates of $f$ by $S_f$.
Let
$\Phi_f:H_1(X_0(N),\Z)\to\Hom_\C(S_f,\C)$ be given by $\Phi$ composed with restriction to $S_f$.
Then $\Phi_f(H_1(X_0(N),\Z))$ is a lattice in $\Hom_\C(S_f,\C)$ and  we have an isomorphism
\[
    A(\C)\cong \Hom_\C(S_f,\C)/\Phi_f(H_1(X_0(N),\Z)).
\]
A choice of basis $\cB$ of $S_f$ induces an isomorphism $\Hom_\C(S_f,\C)\cong
\C^g$ and $\cB$ maps via $\Phi_f$ to a lattice $\Lambda_\cB\subset\C^g$ such that
\[
A(\C)\cong\C^g/\Lambda_\cB.
\]
For a basis $\cB$ of $S_f$ we let $\Lambda^+_\cB$ (resp. $\Lambda^-_\cB$) be the fixed
points of $\Lambda_\cB$ under complex conjugation (resp. under minus complex conjugation).

We define the {\em real period} $\Omega^+_{A}$ (resp. the {\em minus period} $\Omega^-_{A}$) of $A$ as follows:
Let $\omega_A$ be the pullback of a generator of the  global relative
differential~$g$-forms on the N\'eron model $\mathcal{A}$ of $A$ over $\Spec(\Z)$ to $A$.
We call $\omega_A$ a {\em N\'eron differential} on $A$.
Then we define
\[
\Omega^{\pm}_{A}:=\int_{A(\C)^\pm}|\omega_A|.
\]
where $A(\C)^\pm$ denotes the set of  points of $A(\C)$ on which complex conjugation acts as
multiplication by $\pm1$.

Let $\cB$ be a $\Z$-basis of the finitely generated free $\Z$-module consisting of elements of $S_f$ with
Fourier coefficients in $\Z$.
Then we have (see \cite[\S3.2]{agashe-ribet-stein:manin})
\[
\Omega^\pm_{A}=\rho\cdot c_A\cdot\mathrm{vol}(\Lambda^{\pm}_\cB),
\]
where $c_A$ is the Manin constant of $A$, defined, for instance, in
\cite[\S3.1]{agashe-ribet-stein:manin} and $\rho\in\C$ is $i^g$ if $\pm=-$ and~1
otherwise.
It is known that $c_A$ is an integer and conjectured that it is always~1 (cf.
\cite[\S3.3]{agashe-ribet-stein:manin}).
There is no known algorithm to {\em compute} the Manin constant in general, which
complicates much of what we do below.  The evidence that $c_A=1$ is compelling, and we
make the following:

\vspace{1ex}
\begin{center}
{\bf Running Hypothesis:} {\em We assume for the rest of this paper that $c_A=1$.}
\end{center}
\vspace{1ex}

Let $H_1(X_0(N),\Z)^\pm$ denote the part of $H_1(X_0(N),\Z)$ fixed by complex conjugation
(resp. minus complex conjugation).
If $w,z\in\C$, then we write $w\sim z$ if $w$ and $z$ differ by a nonzero rational factor.
We have that
\[
\Omega^\pm_{A}\sim \rho\cdot \mathrm{vol}\left(\widetilde{\Lambda_\cB}^{\pm}\right),
\]
where $\widetilde{\Lambda_\cB}^{\pm}$ is the lattice $\Phi_f(H_1(X_0(N),\Z)^\pm)$,
$\Phi_f$ is induced by the choice of basis $\cB$ as above,
and the rational factor is the number of components of $A(\R)$ (resp. $A(\C)^-$).

We denote the complex $L$-series of $f$ by
\[
L(f,s)=\sum_{n\ge1} \frac{a_n}{n^s}.
\]
If $\psi$ is a Dirichlet character, we denote its Gauss sum by $\tau(\psi)$ and its conjugate
character by $\bar{\psi}$.
We also let $f_\psi$ denote the newform $f$ twisted by $\psi$ and $K_\psi$ the field
generated over $\Q$ by the values of $\psi$.

\begin{theorem}\label{thm:shimura}(Shimura, \cite[Theorem~1]{shimura:periods})
For all $\sigma\in G_f$ there exist $\Omega^+_{f^\sigma}\in\R$ and
$\Omega^-_{f^\sigma}\in i\cdot\R$ such that the following properties are satisfied:
\begin{enumerate}[(i)]
\item We have $$\frac{\pi i}{\Omega^{\pm}_{f^\sigma}}\left(\int_r^{i\infty} {f^\sigma}(z) dz \pm \int_{-r}^{i
\infty} {f^\sigma}(z) dz \right)\in K_f$$ for all $r\in\Q$.
\item If $\psi$ is a Dirichlet character, then
\[
\frac{L(f_{\bar{\psi}},1)}{\tau(\psi)\cdot \Omega^{\sign{\psi}}_f}\in K_f \cdot K_\psi.
\]
\item If $\psi$ is a Dirichlet character, then
\[
\sigma\left(\frac{L(f_{\bar{\psi}},1)}{\tau(\psi)\cdot
\Omega^{\sign(\psi)}_f}\right)=\frac{L(f^\sigma_{\bar{\psi^\sigma}},1)}{\tau(\psi^\sigma)\cdot
\Omega^{\sign(\psi^\sigma)}_{f^\sigma}}.
\]
\end{enumerate}
\end{theorem}
We call a set $\{\Omega^\pm_{f^\sigma}\}_{\sigma\in G_f}$ as in Theorem~\ref{thm:shimura} a set of {\em Shimura periods} for $f$.
Note that the conditions of Theorem~\ref{thm:shimura} do not determine the sets $\{\Omega^\pm_{f^\sigma}\}_{\sigma\in G_f}$.
Indeed, if $\{\Omega^\pm_{f^\sigma}\}_{\sigma\in G_f}$ satisfy the assertions of the
theorem, then this also holds for $\{\sigma(b)\cdot \Omega^\pm_{f^\sigma}\}_{\sigma\in
G_f}$, where $b\in K_f^\times$.

According to Shimura \cite[\S2]{shimura:periods}, the periods $\Omega^\pm_{f^\sigma}$ are
related to a certain period lattice, which gives us a way
to compare them to the periods $\Omega^\pm_A$.
To this end, let $\cB'=(f^\sigma)_{\sigma\in G_f}$ be a basis for
the complex vector space $S_f$, consisting of the Galois conjugates
of $f$ in some order.  Following Shimura, we define an
action of $K_f$ on $\C^g$ as follows:
let $a\in K_f$ act on $\C^g$
via the diagonal matrix $\mathrm{diag}((\sigma_i(a))_i)$.
Let  $(\Lambda_{\cB'}\otimes\Q)^\pm$ be the set of elements of $\Lambda_{\cB'}$ fixed by $\pm$ complex conjugation.
\begin{lemma}[Shimura]\label{lem:shimura}
$(\Lambda_{\cB'}\otimes\Q)^\pm$ is a one-dimensional $K_f$-vector space.
\end{lemma}
\begin{proof}
This is precisely \cite[\S2, pg.~215]{shimura:periods}.
\end{proof}

\begin{theorem}\label{thm:omegas}
Let $\{\Omega^\pm_{f^\sigma}\}_{\sigma\in G_f}$ be any choice of Shimura periods as in
Theorem~\ref{thm:shimura}.
Then we have
\[
    \Omega^\pm_A\sim\prod_{\sigma\in G_f}\Omega^\pm_{f^\sigma}.
\]
\end{theorem}
\begin{proof}
Fix a $\Z$-basis $\cB=(h_1,\ldots,h_g)$  of the free $\Z$-module consisting of elements of $S_f$ with
Fourier coefficients in $\Z$.
Then there are $b_1,\ldots,b_g\in K=K_f$ such that
\[
f=\sum^g_{i=1}b_ih_i.
\]
 If $\sigma\in G_f$, then
\[
f^\sigma=\sum^g_{i=1}\sigma(b_i)h_i
\]
and hence we have
\begin{equation}\label{f_linear}
\langle f^\sigma, \gamma\rangle=\sum^g_{i=1} \sigma(b_i)\cdot \langle h_i,\gamma\rangle
\end{equation}
for each $\gamma\in H_1(X_0(N),\Z)$.

Now fix some ordering ${\sigma_1,\ldots,\sigma_g}$ of the elements of $G_f$ and
let $B=(b_{ij})$ be the $g\times g$-matrix with entries $b_{ij}=\sigma_j(b_i)$,
and $\cB'$ the basis $(f^{\sigma_i})$.
We will compute $\mathrm{vol}\left(\Lambda^\pm_{\cB'}\right)$ in two different ways and
the desired equality up to a rational number will fall out.

First we express $\mathrm{vol}\left(\Lambda^\pm_{\cB'}\right)$ in terms of
$\Omega^\pm_A$.
Note that \eqref{f_linear} implies
\begin{equation}\label{volumes1}
 \rho\cdot\mathrm{vol}\left(\Lambda^\pm_{\cB'}\right)\sim\rho\cdot\mathrm{vol}\left(\widetilde{\Lambda_{\cB'}}^\pm\right)\sim
\rho\cdot|\det(B)|\cdot\mathrm{vol}\left(\widetilde{\Lambda_{\cB}}^\pm\right)\sim|\det(B)|\cdot
\Omega^\pm_A.
\end{equation}

Let $\Omega^\pm\in\C^g$ be the vector whose $i$-th entry is $\Omega^\pm_{f^{\sigma_i}}$.
Since $\det(B)\ne0$, the elements $b_1,\ldots,b_g$ of $K$ form a basis for $K$ over $\Q$, so Lemma~\ref{lem:shimura} implies that
\[
(\Lambda_{\cB'}\otimes\Q)^\pm=K\Omega^\pm.
\]
Thus a basis of
$(\Lambda_{\cB'}\otimes\Q)^\pm$ as a $\Q$-vector space is
$(b_i \cdot \Omega^{\pm} )_{i=1,\ldots,g}$,
where $K$ acts diagonally on $\C^g$, as above.
Hence 
\begin{equation}\label{volumes2}
\rho\cdot\mathrm{vol}\left(\Lambda_{\cB'}^\pm\right)\sim|\det(B)|\cdot\prod_\sigma
\Omega^\pm_{f^\sigma},
\end{equation}
since
\[
(\Lambda_{\cB'}\otimes\Q)^\pm=(\Lambda_{\cB'}^\pm)\otimes\Q.
\]
The proof of the theorem now follows from \eqref{volumes1} and \eqref{volumes2}.
\end{proof}
\begin{remark}
If $L(f,1)\ne0$, then we can also argue as follows:
We have $$\frac{L(A,1)}{\prod_{\sigma\in G_f} \Omega^+_{f^\sigma}}\in\Q$$ by
Theorem~\ref{thm:shimura}.
But on the other hand, the quotient $\frac{L(A,1)}{\Omega^+_A}$ is a rational number as
well by \cite[Theorem~4.5]{agashe-stein:bsd}.
Hence we get $\prod_{\sigma\in G_f} \Omega^+_{f^\sigma}\sim\Omega^+_A$.
\end{remark}
\begin{remark}
In \cite{vatsal:canonical},  Vatsal defines canonical
Shimura periods associated to cuspforms.
It would be interesting to determine whether his periods satisfy
Theorem~\ref{thm:omegas}.
\end{remark}
From now on, we fix some choice $\{\Omega^\pm_{f^\sigma}\}_{\sigma\in G_f}$ such that
\begin{equation}\label{norm_periods}
\prod_{\sigma\in G_f}\Omega^\pm_{f^\sigma}=\Omega^\pm_A.
\end{equation}
If $\{\Psi^\pm_{f^\sigma}\}_{\sigma\in G_f}$ is another set of Shimura periods satisfying
\eqref{norm_periods}, then there is a unit $b\in\O_{K_f}$ such that
$\Psi^\pm_{f^\sigma}=\sigma(b)\Omega^\pm_{f^\sigma}$ for all $\sigma\in G_f$.
For our intended applications, this ambiguity is not serious, see Remark~\ref{rk:l_fn_A}.

In order to {\em compute} $\{\Omega^\pm_{f^\sigma}\}_{\sigma\in G_f}$ we can find a
Dirichlet character $\psi$ of sign $\pm$ such that $L(f_\psi,1)\ne 0$ and use equation~(11) of
\cite{shimura:periods}.
Alternatively, we can fix some nonzero element $\gamma\in H_1(X_0(N),\Z)[I_f]^\pm$ and define
\[
\Omega^\pm_{f^\sigma}=\sigma(b)\cdot\langle f^\sigma, \gamma\rangle,
\]
for each $\sigma\in G_f$, where $b\in K_f$ is chosen to make~\eqref{norm_periods} hold.
See \cite[Ch.~10]{stein:modform} for a description of how to compute the integration pairing in
practice.

As an application of Theorem~\ref{thm:omegas}, we can prove a relation between the
real and minus period of $A$ and the corresponding periods of $A$ twisted by a
Dirichlet character $\psi$.
\begin{corollary}\label{cor:characters}
Let $\psi$ be a Dirichlet character such that $L(f_\psi,1)\ne0$.
Then there exists $\eta_\psi\in K^*_\psi$ such that
\[
    \Omega^{\sign{\psi}}_{A}\cdot\prod_{\sigma\in
    G_f}\tau\left(\psi^\sigma\right)=\eta_{\psi}\cdot\Omega^+_{A_\psi}.
\]
In particular, if $\psi$ takes values in $\Q$, then there exists $\eta_\psi\in\Q^*$ such
that
\[
    \Omega^{\sign{\psi}}_{A}\cdot\tau\left(\psi\right)^g=\eta_{\psi}\cdot\Omega^+_{A_\psi}.
\]
\end{corollary}
\begin{proof}
This follows from Theorem~\ref{thm:omegas} and Theorem~\ref{thm:shimura}.
\end{proof}
\begin{remark}
If $A$ is the Jacobian of a hyperelliptic curve of genus at most~2 and $\psi$ is a quadratic Dirichlet
character such that $\psi(N)\ne 0$, then one can show that the statement of
Corollary~\ref{cor:characters} holds without the assumption $L(f_\psi,1)\neq 0$ using quite
concrete arguments.
For elliptic curves and quadratic $\psi$, Corollary~\ref{cor:characters} was already used
in \cite[\S II.11]{mtt}.
Note, however, that their
claim that $\eta_{\psi}\in\{1,2\}$ is incorrect; see \cite{pal:periods}, where the correct
value of $\eta_\psi$ is determined in all cases.
\end{remark}

   \subsection{Modular symbols, measures, and the $p$-adic  $L$-function of a newform}
In this subsection we define the $p$-adic $L$-function associated to $f$, following
\cite{mtt}.
See also the treatment in \cite{pollack:supersingular}.
The definitions for $f^\sigma$, where $\sigma\in G_f$, are entirely analogous.

Recall that we fixed a choice of Shimura periods $\Omega_{f^\sigma}^{\pm}$ above.
The \emph{plus modular symbol map} associated to $f$ is the map
      \begin{align*}[\;]^+_{f}: \Q &\ra K_f\\
      r &\mapsto [r]^+_{f} = -\frac{\pi
i}{\Omega^+_{f}}\left(\int_r^{i\infty} {f}(z) dz + \int_{-r}^{i \infty}
{f}(z) dz \right)\end{align*}
and the \emph{minus modular symbol map} associated to $f$ is the map
      \begin{align*}[\;]^-_{f}: \Q &\ra K_f\\
      r &\mapsto [r]^-_{f} = \frac{\pi
i}{\Omega^-_{f}}\left(\int_r^{i\infty} {f}(z) dz - \int_{-r}^{i \infty}
{f}(z) dz \right).\end{align*}

Note that we have $[0]^{+}_{f} = \frac{L({f},1) }{\Omega^+_{f}}$.
More generally, if $m$ is a positive integer and $\psi$ is a Dirichlet character modulo
$m$, then \cite[\S I.8]{mtt} implies
\begin{equation}\label{twist_interp}
\frac{L(f_{\bar{\psi}},1)}{\Omega^{\sign{\psi}}_f}=\frac{\psi(-1)}{\tau(\psi)}\sum_{u\bmod{m}}\psi(a)\cdot\left[\frac{u}{m}\right]^{\sign{\psi}}_f\in K_f.
\end{equation}

      Let $p$ be a prime of good ordinary reduction for $A$.
We fix, once and for all, a prime $\wp$ of $K_f$ lying above $p$.
The modular symbol maps allow us to define two measures on $\Z_p^{\times}$ which depend on
the unit root of the polynomial $h(x) := x^2 - a_p x + p \in (K_f)_{\wp}[x]$, where $(K_f)_\wp$ is
the completion of $K_f$ at $\wp$. The construction of the $p$-adic $L$-function depends, in turn, on these measures.
Since $A$ is ordinary at $p$, the polynomial $h$ has a unique unit
root $\alpha \in (K_f)_{\wp}$, i.e., a root with $\ord_\wp(\alpha) = 0$.

Using the modular symbol maps $[\;]^\pm_{f}$, we define two measures
$\mu^{\pm}_{{f},\alpha}$ on
$\Z_p^{\times}$
by $$\mu^{\pm}_{{f},\alpha}(a + p^n \Z_p) =
\frac{1}{\alpha^n}\left[\frac{a}{p^n}\right]^\pm_{f} -
\frac{1}{\alpha^{n+1}}\left[\frac{a}{p^{n-1}}\right]^\pm_{f}.$$

      For a continuous character $\chi$ on $\Z_p^{\times}$ with values in $\C_p$, we may
integrate $\chi$ against $\mu_{f,\alpha}$.  Following \cite[\S I.13]{mtt}, we write $x \in \Z_p^{\times}$ as $\omega(x) \cdot \langle x \rangle$ where $\omega(x)$ is a $(p-1)$-st root of unity and $\langle x \rangle$ belongs to $1 + p \Z_p$. The element $\omega$ is known as the Teichm\"{u}ller character.

      We define the analytic $p$-adic $L$-function associated to $f$ by
\begin{equation}\label{intchar}L_p(f,s) = \int_{\Z_p^{\times}}  \langle x \rangle^{s-1} \,
d\mu^+_{f,\alpha}(x) \qquad\text{for all}\; s \in \Z_p,\end{equation}
where by $\langle x \rangle^{s-1}$ we mean $\exp_p((s-1)\cdot \log_p \langle x \rangle$
and $\exp_p$ and $\log_p$ are the $p$-adic exponential and logarithm, respectively.  The
function $L_p(f,s)$ extends to a locally analytic function in $s$ on the disc defined by $| s - 1| < 1$, as in the first proposition of  \cite[\S I.13]{mtt}.

Let ${}_{\infty}G$ be the Galois group $\Gal(\Q(\mu_{p^{\infty}})/\Q)$. The cyclotomic character
$\kappa:{}_{\infty}G \ra \Z_p^{\times}$ induces an isomorphism from ${}_{\infty}G$
to $\Z_p^{\times}$ that sends a topological generator $\gamma$ in ${}_{\infty}G^{(p-1)}$
to a generator $\kappa(\gamma)$ of $1 + p\Z_p^{\times}$. This identification allows us to
give a series expansion of the $p$-adic $L$-function in terms of $T = \kappa(\gamma)^{s-1} - 1$. That is, we have
\begin{equation}\label{series}\mathcal{L}_p(f,T) = \int_{\Z_p^{\times}}
(1+T)^{\frac{\log_p(\langle x \rangle)}{\log_p(\kappa(\gamma))}} \,d\mu^+_{f,\alpha}(x).\end{equation}

 Now for each $n \geq 1$, let $P_n(f,T)$ be the following polynomial:

\begin{equation}\label{rs}P_n(f,T) = \sum_{a=1}^{p-1} \left( \sum_{j=0}^{p^{n-1}-1}
\mu^+_{f,\alpha}   \left(\omega(a)(1+p)^j + p^n\Z_p\right) \cdot (1+T)^j \right)  .\end{equation}

We have that \eqref{rs} gives us a Riemann sum for the integral \eqref{series}, by summing over residue classes mod $p^n$; in other words:

 \begin{proposition}
        We have that the $p$-adic limit of these polynomials is the $p$-adic
        $L$-series:
        $$
           \lim_{n\to\infty} P_n(f,T) = \cL_p(f, T).
        $$
        \end{proposition}
        This convergence is coefficient-by-coefficient, in the sense that
        if
        $P_{n}(f,T) = \sum_j a_{n,j} T^j$ and
        $\cL_p(f,T) = \sum_j a_j T^j$, then
        $$
        \lim_{n \to \infty} a_{n,j} = a_j.
        $$
        \begin{proof}This is a straightforward generalization of \cite[Proposition~3.1]{stein-wuthrich:shark}.
The upper bounds we obtain are the same as the upper bounds
in the proof of \cite[Proposition~3.1]{stein-wuthrich:shark}, which enables us to compute
the $p$-adic $L$-series to any desired precision.
\end{proof}
We define the {\em $p$-adic multiplier} $\epsilon_p(f)$ by
\[
\epsilon_p(f)=\left(1-\alpha^{-1}\right)^2.
\]

The $p$-adic $L$-series of $f$ satisfies an interpolation property with
respect to the complex $L$-series of $f$ \cite[\S I.14]{mtt}:
\begin{equation}\label{interp_f}\mathcal{L}_p(f,0) = L_p(f,1) =  \int_{\Z_p^{\times}}  \,
d\mu^+_{f,\alpha} = \epsilon_p(f)\cdot[0]^+_f=
\epsilon_p(f) \cdot \frac{L(f,1)}{\Omega^+_f}.\end{equation}

        \subsection{$p$-adic $L$-function associated to $A$}\label{sec:abelian}
The $p$-adic $L$-series $\cL_p(f,s)$ associated to $f$ that we constructed in the
previous subsection depends on the Shimura periods $\Omega^+_{f^\sigma}$ and on the prime $\wp$.
In the present section we define a
$p$-adic $L$-function associated to the abelian variety $A$ which is independent of the
choices of the Shimura periods (provided they satisfy~\eqref{norm_periods}) and of $\wp$.

The abelian variety $A $ has an associated complex $L$-series, given by $$L(A,s) =
\prod_{{\sigma\in G_f}} L(f^\sigma,s),$$
which can be extended analytically to the whole complex plane.

We define the {\em $p$-adic $L$-function associated to $A$} by
\begin{equation}\label{Lpf}
L_p(A,s)=\prod_{\sigma\in G_f} L_p(f^\sigma,s)
\end{equation}
for $s\in\Z_p$.
\begin{remark}\label{rk:l_fn_A}
    Since we require our Shimura periods $\{\Omega^\pm_{f^\sigma}\}_{\sigma\in G_f}$ to satisfy
    \eqref{norm_periods}, the $p$-adic $L$-function $L_p(A,s)$ does not depend on the
    choice of period for each $f^\sigma$, although the individual $p$-adic $L$-functions
    $L_p(f^\sigma, s)$ do.
\end{remark}
Furthermore, we define
\begin{equation}\label{cLpf}
\cL_p(A,T)=\prod_{\sigma\in G_f} \cL_p(f^\sigma,T)
\end{equation}
and the {\em $p$-adic multiplier of $A$} by
$$\epsilon_p(A)=\prod_{\sigma\in G_f} \epsilon_p(f^\sigma).$$
For $r\in\Q$ we set
\begin{equation}\label{modsym_abvar}[r]^\pm_A:=\prod_{\sigma\in G_f}[r]^\pm_{f^\sigma}.\end{equation}
The following result follows immediately from~\eqref{interp_f}:
\begin{corollary}Let $\{\Omega^\pm_{f^\sigma}\}_{\sigma\in G_f}$ be a set of Shimura periods satisfying \eqref{norm_periods}.
  Then we have
\begin{equation}\label{interpol}
    \mathcal{L}_p(A,0) = L_p(A,1) =  \epsilon_p(A)\cdot[0]^+_A= \epsilon_p(A) \cdot
\frac{L(A,1)}{\Omega^+_A}.
\end{equation}
\end{corollary}

    \subsection{Quadratic twists and normalization}\label{sec:norm}

    Modular symbols can be computed up to a rational multiple purely
    algebraically (cf. \cite{stein:modform}) using (mostly sparse)
    linear algebra over fields.  Computing the exact modular symbol
    (not just up to a rational factor) requires doing linear
    algebra over $\Z$, which is much slower.  In this section we
    describe a method to determine the correct normalization of the
    modular symbol map by using special values of quadratic twists,
    which is potentially much faster than using linear algebra over
    $\Z$.

    In order to find the correct normalization, we use the fact that the $p$-adic
$L$-series associated to $A$ interpolates the
    Hasse-Weil $L$-function $L(A, s)$ associated to $A$ at special values.
    Algorithms for the computation of $\frac{L(A,1)}{\Omega^+_A},\, L(A, 1)$ and $\Omega^+_A$
    are discussed in \cite{agashe-stein:bsd, empirical, stein:modform}.
    So if $L(A,1)\ne 0$, then we can find the correct normalization factor
    $\delta^+$ for $\cL_p(A,T)$  by computing $[0]^+_A$ and comparing it to
    $\frac{L(A,1)}{\Omega^+_A}$.  Note that the quotient $\frac{L(A,1)}{\Omega^+_A}$ can
    be computed purely algebraically as a certain lattice index, without computing either of the real
    numbers $L(A,1)$ or $\Omega^+_A$.

In order to discuss the strategy for the case $L(A,1)=0$, we begin by considering
modular symbols associated to quadratic twists of $f$.
     Let $D$ be a fundamental discriminant of a quadratic number field such that $\gcd(pN, D)
    =1$ and let $\psi$ denote the Dirichlet character associated to $\Q(\sqrt{D})$.
We assume that $L(f_\psi,1)\ne 0$.
    Using Corollary~\ref{cor:characters} we see that there exists
$\eta_\psi\in K_f^*$ such that
\[\Omega^{\sign(D)}_{A}\cdot D^{g/2}=\eta_{\psi}\cdot\Omega^+_{A_\psi}.\]
    A computation analogous to \cite[\S3.7]{stein-wuthrich:shark} yields
    \begin{equation}\label{twisted_symbol}
[r]^+_{A_\psi}=    \prod_{\sigma\in
G_f}[r]^+_{f^\sigma_\psi}=\frac{\sign(D)^g}{\eta_{\psi}}\prod_{\sigma\in
G_f}\sum^{|D|-1}_{u=1}\psi(u)\cdot\left[r+\frac{u}{D}\right]^{\sign(D)}_{f^\sigma}.
    \end{equation}
    Therefore we can compute the product of the plus modular symbols for $f_\psi$ and its
conjugates in terms of modular symbols for  $f$ and its conjugates.
    The same holds for the $p$-adic $L$-function of the twist $A_\psi$ of $A$ by $\psi$.

Now suppose that $\cL_p(A, 0)=0$ and we want to find the correct normalization factor for
$[\;]^+_f$ .
We can use that for a fundamental discriminant $D$ with Dirichlet character $\psi$, the modular
    symbols $[\;]^+_f$ and $[\;]^{\sign(D)}_{f_\psi}$ are related by \eqref{twisted_symbol}.
      Hence the same normalization factor $\delta^+$  will yield the correct value
$[r]^+_{A_\psi}$ for all $D>0$ such that $\gcd(pN, D)=1$.

    We can compute $\delta^+$ by finding a fundamental discriminant $D>0$ such that $\gcd(pN,
    D)=1$ and such that $A_\psi$ has analytic rank~0 over $\Q$, where $\psi$ is the quadratic
character associated to $\Q(\sqrt{D})$, and comparing $ [0]^+_{A_\psi}$
    to
    \[
        \frac{\eta_\psi\cdot L(A_\psi,1)}{D^{g/2}\cdot\Omega^+_A}.
    \]
    An analogous approach can be used to find the correct normalization factor $\delta^-$ for
    the minus modular symbol.
It follows from \cite{bump-friedberg-hoffstein} that in both cases a
fundamental discriminant $D$ as above always exists.
    \begin{remark}\label{rk:find_eta}
    Suppose that $A$ is the Jacobian of a hyperelliptic curve $X/\Q$ of genus $g$ given by an equation
    \[
    y^2+h_1(x)y=f_1(x)
    \]
    which is minimal in the sense of \cite{liu:entiers}.
    Often it is not necessary to compute $\Omega^{\sign D}_{A_\psi}$ (or even $\Omega^+_A$) to
    compute $\eta_\psi$.
Let $  y^2+h_2(x)y=f_2(x)$ be a minimal equation for $X_\psi$,
and consider the differentials $\omega_i=\frac{x^idx}{2y+h_1(x)}$ on $X$
and $\omega'_i=\frac{x^idx}{2y+h_2(x)}$ on $X_{\psi}$.
 It frequently
happens that $(\omega_1,\ldots,\omega_g)$ is a basis of the integral~1-forms on $A$ and
$(\omega'_1,\ldots,\omega'_g)$ is also a basis for the integral 1-forms on $A_\psi$.
In that case we always have $\eta_\psi\in\{\pm1\}$; this follows from
\cite[\S3.5]{empirical}.
It is easy to determine the sign using a straightforward generalization of
\cite[\S1.3]{liu:conducteur}.
More generally, a similar approach can also be used to compute $\eta_\psi$ directly if we
know how to express a basis for the integral~1-forms in terms of
$\omega_1,\ldots,\omega_g$.
    \end{remark}

    \subsection{The algorithm}\label{sec:algorithm}
    We implemented the following algorithm for the computation of the $p$-adic $L$-series of $A$ in {\tt Sage} \cite{sage}.
    \begin{algorithm}[$p$-adic $L$-series]\label{alg:padic_lseries}$\quad$\\
    \textbf{Input:} Good ordinary prime $p$, $A$ modular abelian variety attached to newform $f$, precision $n$.\\
    \textbf{Output:} $n$th approximation to the $p$-adic $L$-series $\mathcal{L}_p(A,T)$.
    \begin{enumerate}
    \item Fix a prime $\wp$ of the field $K_f$ generated by the Hecke eigenvalues of $f$ lying above $p$ and compute the unit root $\alpha$ of $h(x)\in (K_f)_\wp[x]$.
    \item Find a fundamental discriminant $D>0$ such that $\gcd(pN,D)=1$ and $A_\psi$ has
analytic rank~0     over $\Q$, where $\psi$ is the quadratic character associated to $\Q(\sqrt{D})$.
    \item Compute $\eta_\psi$.
    \item For each $\sigma\in G_f$, define measures $\mu^{\pm}_{{f^\sigma},\alpha}$.
    \item For each $\sigma\in G_f$, compute $[0]^+_{f^\sigma_\psi}$ and set
$[0]^+_{A_\psi}=\prod_{\sigma\in G_f}[0]^+_{f^\sigma_\psi}$.
    \item Compute $\frac{L(A_\psi,1)}{D^{g/2}\cdot\Omega^+_A}$ and deduce the normalization factor $\delta^+$ using \eqref{interpol}.
\item For each $\sigma$, compute $P_{n}(f^\sigma,T)$.
    \item Return $\delta^+\cdot\prod_\sigma P_n(f^\sigma,T)$.
        \end{enumerate}
    \end{algorithm}

        \begin{remark}Note that Step (7) of Algorithm~\ref{alg:padic_lseries} is exponential in $p$; see the following subsection for an alternative method.\end{remark}

    \subsection{Overconvergent modular symbols}\label{sec:OMS}
   Here we outline an alternative method for Step (7) of Algorithm~\ref{alg:padic_lseries}.
    This method is due to Pollack and Stevens \cite{pollack-stevens:overconv} and has
    running time polynomial in $p$ and in the desired number of digits of
    precision.

    The idea is to use Stevens's overconvergent modular symbols; these
    are constructed using certain $p$-adic distributions, and they can
    be specialized to classical modular symbols.  More precisely, any
    classical modular eigensymbol can be lifted uniquely to an
    overconvergent modular Hecke-eigensymbol, which can be
    approximated using finite data.
    Note that in order to do this, we first have to $p$-stabilize the symbol
    to a symbol for $\Gamma_0(Np)$ which is an eigensymbol away from $p$.

    The plus modular symbol we start
    with is only determined up to multiplication by a scalar, so the
    corresponding overconvergent eigenlift is also only determined up
    to multiplication by a scalar.
    Hence we cannot dispense with Steps (2), (3), (5) and (6) of
    Algorithm~\ref{alg:padic_lseries}.

    Once this desired lift has been computed, writing down the $p$-adic $L$-series
    associated to the modular symbol and its quadratic twists by $\psi$ for suitable $D$ is rather easy, cf.
    \cite[\S9]{pollack-stevens:overconv}.
    Together with David Roe and Robert Pollack, we have implemented Algorithm~\ref{alg:padic_lseries} with Step (7) replaced by the
algorithm from \cite{pollack-stevens:overconv} in {\tt Sage} as well, building on an implementation
    due to Pollack.

    \section{The $p$-adic height pairing of Coleman and Gross and $p$-adic regulators}
We now shift our attention to the remaining $p$-adic quantity appearing in Conjecture~\ref{pbsd}, the $p$-adic regulator. To discuss $p$-adic regulators, we begin, in this section, by describing one construction of the global $p$-adic height pairing relevant to our setting. We give an algorithm to compute the height pairing in the case when the abelian variety $A$ is the Jacobian of a hyperelliptic curve and show how we use it to compute $p$-adic regulators.

Let $A$ be an abelian variety defined over a number field $K$ and
let $A^{\vee}$ denote the dual abelian variety to $A$.
There are several definitions of $p$-adic height pairings on abelian varieties in the
literature. The work of Schneider \cite{schneider:height1} and Mazur and Tate \cite{mazur-tate:canonical}
gave the first constructions of $p$-adic height pairings on abelian varieties defined over
number fields. This was extended to motives by Nekov{\'a}{\v{r}} \cite{MR1263527}. There
are also more specialized definitions: in the case when $\dim A = 1$, $K = \Q$, and $p$ is
a prime of good, ordinary reduction, Mazur, Stein, and Tate \cite{mazur-stein-tate:padic}
gave an explicit formula for the $p$-adic height which relies on an understanding of the
$p$-adic sigma function. When $A$ is the Jacobian of a curve, Coleman and Gross
\cite{coleman-gross} described the $p$-adic height pairing on $A$ as a sum of local height
pairings.  Note that in the range where all of these constructions apply, they are known
to be equivalent by the work of Coleman \cite{MR1091621} and Besser \cite{Bes99a} (where the equivalence is possibly up to sign, e.g., in the supersingular case).  For all of these definitions, the $p$-adic height pairing is known to be bilinear and,
     in the principally polarized case, symmetric.

Let $p$ be a prime number such that $A$ has good ordinary reduction at all primes of $K$
above $p$.  The canonical $p$-adic height pairings are (up to nontrivial scalar multiple) in one-to-one correspondence with certain $\Z_p$-extensions $L/K$: those $L$ with finitely many ramified primes which are primes of ordinary reduction for $A$. For example, when $K = \Q$, there is, up to nontrivial scalar multiple, one canonical $p$-adic height pairing, the \emph{cyclotomic} $p$-adic height.

Fix a $p$-adic height pairing 
\begin{align*} h: A \times A^{\vee} &\longrightarrow \Q_p \\
     (P, Q) &\mapsto h(P,Q).\end{align*}
Now we want to define the $p$-adic regulator with respect to $h$.
In the literature, one usually defines this quantity as the determinant of the height
pairing matrix with respect to a set of generators of the free part of $A(K)$ and
$A^\vee(K)$, respectively.
This, however, is only well-defined up to sign.
Since in Iwasawa theory one is typically only interested in results up to a $p$-adic unit, this is usually
not a serious problem, but in order to state Conjecture~\ref{pbsd}, we need a canonical well-defined $p$-adic
regulator.
First note that if $\phi:A\to A^\vee$ is an isogeny and $P_1,\ldots,P_r \in A(K)$ map to a
basis of $A(K)/\tors$ then
$$\Reg_{p,\phi}(A/K) := \det\left(\left(h(P_i,\phi(P_j))\right)_{i,j}\right)$$ does not depend on the choice of
$P_1,\ldots,P_r$.
Also recall that if $c \in \Pic_{A/K}$ is ample, then the map $\phi_c :A \to \Pic^0_{A/K} \cong
A^\vee$ which
maps $P \in A$ to $t_P^*(c) \otimes c^{-1}$, where $t_P$ is the translation-by-$P$ map, is an isogeny.
We first proof an elementary lemma.
\begin{lemma}\label{lem:oriented}
  Let $V$ and $W$ be finite-dimensional $\R$-vector spaces
  equipped with a non-degenerate $\R$-bilinear pairing $B:V\times W \to \R$.
  Let $\psi, \psi' : V \rightrightarrows W$ be isomorphisms
  such that the pullbacks $b$ (resp. $b'$) of $B$ along $1 \times \psi$ (resp. $1 \times
  \psi'$) are symmetric and positive-definite.
  Then $\psi^{-1}\circ \psi' \in \GL(V)$ has positive determinant.
\end{lemma}
\begin{proof}
  Setting $\theta := \psi^{-1}\circ \psi'$, we have that $b'$ is the pullback of $b$ along $1 \times \theta$.
  Because $b'$ is symmetric, $\theta$ is self-adjoint with respect to $b$.
  The spectral theorem implies that $\theta$ is diagonalizable over $\R$.
  All eigenvalues are positive because $b'$ is positive definite, which proves the lemma.
\end{proof}

\begin{corollary}\label{cor:reg}
  Let $c,c'\in \Pic_{A/K}(K)$ be ample and symmetric and let
  $\phi_c,\phi_{c'}:A\rightrightarrows A^\vee$ be the corresponding
    isogenies.
    Set $m=[A^\vee(K)/\tors:\phi_c(A(K)/\tors)]$ and
$m'=[A^\vee(K)/\tors:\phi_{c'}(A(K)/\tors)]$.
    Then we have
    $$\frac{1}{m}\Reg_{p,\phi_c}(A/K)=
    \frac{1}{m'}\Reg_{p,\phi_{c'}}(A/K).$$
\end{corollary}
\begin{proof}
  It is clear that the claimed equality holds up to sign.
  Hence it suffices to show that if $P_1,\ldots,P_r\in A(K)$ (resp.  $Q_1,\ldots,Q_r \in
  A^\vee(K)$) map to generators of $A(K)/\tors$ (resp. $A^\vee(K)/\tors)$), and if
$$\phi_c(P_i)=\sum^r_{j=1}m_{ij}Q_j\quad\textrm{and}\quad
\phi_{c'}(P_i)=\sum^r_{j=1}m'_{ij}Q_j,$$
then the determinants of the matrices $\left(\left(m_{ij}\right)_{1\le i,j\le
r}\right)$ and $\left(\left(m'_{ij}\right)_{1\le i,j\le r}\right)$
have the same sign.
In other words, it suffices to show that $\psi^{-1}\circ\psi'$ has positive determinant, where
$\psi$ (resp. $\psi'$) is $\phi_c$ (resp. $\phi_{c'}$), extended to an isomorphism
$A(K)\otimes\R\rightarrow A^\vee(K)\otimes\R$.
But by the proof of~\cite[Theorem~B.5.8]{hindry-silverman:diophantine},
the pullback of the canonical N\'eron-Tate height pairing $A(K)\times A^\vee(K) \to \R$ along $1\times
\phi_c$ (resp. $1\times \phi_{c'}$) is the N\'eron-Tate height pairing $A(K) \times A(K)
\to \R$ with respect to $c$ (resp. $c'$).
Therefore we can apply Lemma~\ref{lem:oriented} to $\psi$ and $\psi'$ and the result follows.

\end{proof}

\begin{definition}\label{def:reg}Let  $A$ be an abelian variety defined over a number field $K$ and
let $A^{\vee}$ denote its dual.
Fix some ample and symmetric $c\in\Pic_{A/K}(K)$ and let $\phi_c:A\to A^\vee$ denote the corresponding isogeny.
The \emph{$p$-adic regulator}  of $A$, denoted
$\Reg_p(A/K)$, is defined by
$$\Reg_p(A/K):=\frac{1}{[A^\vee(K)/\tors:\phi_c(A(K)/\tors)]}\Reg_{p,\phi_c}(A/K)$$
 \end{definition}
 The $p$-adic regulator is well-defined by Lemma~\ref{cor:reg}.
         It has been conjectured by Schneider \cite{schneider:height1} that the $p$-adic
     height pairing is nondegenerate, but in contrast to the classical case of
     N\'eron-Tate heights this is not known in general.

Among the aforementioned definitions of the $p$-adic height pairing, the Coleman and Gross construction of the $p$-adic height pairing is fairly explicit in nature, and for that reason, lends itself nicely to computation.
Thus we take it as our working definition of the $p$-adic height. We start by giving a brief overview of the work of Coleman and Gross.

   Suppose $X/K$ is a curve defined over a number field $K$, with
    good reduction at primes above $p$.
    To define the $p$-adic height pairing
    \begin{equation*}
      h: \Div^0(X) \times \Div^0(X) \to \Q_p\;,
    \end{equation*}
    where $\Div^0(X)$ denotes the divisors on $X$ of degree zero, one needs the following data:
    \begin{itemize}
    \item  A ``global log''- a continuous idele class character $\ell: \mathbb{A}_K^\ast/K^\ast \to \Q_p$\;.
    \item For each $v\mid{}p$ a choice of a subspace $W_v\subset\hdr^1((X\otimes
      K_v)/K_v)$ complementary to the space of holomorphic forms.
    \end{itemize}
  We require that the
    local characters $\ell_v$ induced by $\ell$, for $v\mid{}p$, are ramified
    in the sense that they do not vanish on the units in $K_v$.  From $\ell$ one deduces the following data:
    \begin{itemize}
    \item For any place $v\ndiv p$ we have $\ell_v(\O_{K_v}^\ast)=0$ for
      continuity reasons, which implies that $\ell_v$ is completely
      determined by the number $\ell_v(\pi_v)$, where $\pi_v$ is any
      uniformizer in $K_v$.
    \item For any place $v\mid{}p$ we can decompose $\ell_v$ as a composition
      \begin{equation}\label{tdefnd}
        \xymatrix{
          {\O_{K_v}^\ast}  \ar[rr]^{\ell_v} \ar[dr]^{\log_v} & &   \Q_p\\
          & K_v\ar[ur]^{t_v}
        }
      \end{equation}
      where $t_v$ is a $\Q_p$-linear map. Since we assume that $\ell_v$
      is ramified it is then possible to extend $\log_v$ to $ \log_v : K_v^\ast \to K_v$ in such a way that the diagram remains commutative.
    \end{itemize}

We will later need to choose a branch of the $p$-adic logarithm, since the Coleman integral of a form with residue depends on such a choice. We will fix this choice for the computation of the local height pairing to be the one determined above.

    Let us now describe the $p$-adic height pairing $h(D,{E})$ for a pair of degree zero
    divisors ${D}$ and ${E}$ with disjoint support. The height pairing is a sum of local
    terms $$h({D},{E})= \sum_v h_v({D},{E})$$ over all finite places
    $v$. The local terms depend only on the completion at $v$ of $K$. Thus,
    let $K_v$ be the completion of $K$ at a place $v$, with valuation ring
    $\cO$, uniformizer $\pi$ and let $k_v = \cO/\pi\cO$ be the residue
    field, with order $q.$ Let $C$ denote the curve $X$ over the local field $K_v$. We shall assume that $C$ has a $K_v$-rational
    point and that $C$ has good reduction at $\pi$.

    \begin{proposition}\label{intersect} If $\chr k_v \neq p$, there exists a unique function $\langle D, E \rangle$ defined for all $D,E
      \in \Div^0(C)$ of disjoint support that is continuous, symmetric,
      bi-additive, takes values in $\Q_p$, and satisfies
      \begin{equation}
      \langle (f),  E\rangle = \ell_v(f(E))\label{heighprinc}
    \end{equation}
     for $f \in K_v(C)^*$.
    \end{proposition}

    \begin{proof}See \cite[Prop 1.2]{coleman-gross}.
    \end{proof}
    We will discuss how to compute this function in practice in
    Section~\ref{sec:local_away_from_p}.

    \subsection{Computing $p$-adic heights away from $p$}\label{sec:local_away_from_p}
    We keep the notation of the previous section, but assume, in addition, that $X$
    is hyperelliptic of genus $g$, given by an equation
     $y^2=f(x)$, where $f\in\cO[x]$ is separable.
    Let $v$ be a fixed non-archimedean place of $K$ not dividing $p$.

    The arithmetic geometry needed in the present section can be found in
    \cite[Chapters~8,9]{liu:agac}.
    We fix a proper regular model $\mathcal{C}$ of
    $C=X\times_K K_v$ over $\Spec(\cO)$ with special fiber
    $\mathcal{C}_v$. If $D$ is a prime
    divisor on $C$, then we let $\overline{D}$ denote the
    Zariski closure of $D$ on $\mathcal{C}$ and we extend this
    to all of $\mathrm{Div}(C)$ by linearity.
    It was shown by Hriljac \cite{hriljac:heights} that if $D\in\Div^0(C)$, then there exists a vertical $\Q$-divisor
    $\Phi(D)$ on $\mathcal{C}$ such that the intersection multiplicity of
$\overline{D}+\Phi(D)$ with any irreducible component of $\mathcal{C}_v$ is trivial.

    If $D, E \in \Div^0(C)$ have disjoint support, then according to
    \cite[Prop 1.2]{coleman-gross}
    the local height pairing between
    $D$ and $E$ at $v$ is given by
    \begin{equation}\label{htdef}
      h_v(D,E)=\ell_v(\pi_v)\cdot
    i_v\left(\overline{D}+\Phi(D),\overline{E}\right),
    \end{equation}
    where $i_v$ denotes the (rational-valued) intersection
    pairing on $\mathcal{C}$.
    This does not depend on the choice of $\Phi(D)$ or of $\mathcal{C}$.

    As in \cite{mueller:computing}, the following
    steps are sufficient to compute the local $p$-adic height
    pairing at $v$.
    \begin{itemize}
            \item[(1)] Compute a desingularization $\mathcal{C}$ in the strong
    sense of the Zariski closure $\overline{C}$ of $C$ over $\Spec(\cO)$;
      \item[(2)] Compute
    $i_v\left(\overline{D},\overline{E}\right);$
      \item[(3)] Compute
    $i_v\left(\Phi(D),\overline{E}\right)$.
    \end{itemize}
    These steps are dealt with in detail and greater generality in
    \cite{mueller:computing}.
    For the convenience of the reader, we provide a brief
    summary in the present case of hyperelliptic curves.

    Step (1) can be done using a desingularization algorithm
    implemented by Steve Donnelly in {\tt Magma} \cite{magma}.
    Recall that a desingularization $\mathcal{C}$ of the
    Zariski closure $\overline{C}$ of $C$ over $\Spec(\cO)$ in the strong
    sense is a proper regular model of $C$ over $\Spec(\cO)$ such that
    there exists a morphism $\xi:\mathcal{C}\ra \overline{C}$ that is an
    isomorphism above regular points of $\overline{C}$.
    From now on we will assume that our model $\mathcal{C}$ is of this
    type, as this property is needed in order for some of the other
    steps to work.
    See \cite[\S4.3]{mueller:computing} for details.

    For Step (2), we write our divisors $D$ and $E$ as
    differences of effective divisors
    \[
      D=D_1-D_2,\quad E=E_1-E_2.
    \]
    By bilinearity of the intersection pairing it suffices to
    discuss the computation of
    $i_v\left(\overline{D_1},\overline{E_1}\right)$.

    For now we assume that the points on $\mathcal{C}_v$ where $\overline{D_1}$ and
    $\overline{E_1}$ intersect all lie on a single affine piece $\mathcal{C}^a$ of
    $\mathcal{C}$.
    Suppose that $\mathcal{C}^a=\Spec(\cO[x_1,\ldots,x_n]/J)$ for some ideal $J$ and that
    $I_{\overline{D_1}}$ (resp. $I_{\overline{E_1}}$) represents $\overline{D_1}$ (resp. $\overline{E_1}$)
    on $\mathcal{C}$.
    Then we have
    \begin{equation}\label{IntForm}
      i_v\left(\overline{D_1},\overline{E_1}\right)=
            \mathrm{length}_{\cO_{\mathcal{C}^a_v}}\left(\left(\cO[x_1,\ldots,x_n]/J+I_{\overline{D_1}}+I_{\overline{E_1}}\right)_{(\pi_v)}\right).
    \end{equation}
    The computation of the right hand side of \eqref{IntForm} can be reduced to
    (essentially) the computation of Gr\"obner bases over $\Spec(\cO)$, cf.
    \cite[Algorithm~1]{mueller:computing}.

    In order to find the representing ideals $I_{\overline{D_1}}$ and
    $I_{\overline{E_1}}$ the strategy is
    to first find representing ideals for the Zariski closures of $D_1$ and $E_1$ on
    $\overline{C}$ and lift these to $\mathcal{C}$ through the blow-up process.
    We can guarantee that the intersection of these closures has support only in one
    of the two standard affine pieces of $\overline{C}$ by decomposing $D_1$ and
    $E_1$ into prime divisors over a finite extension $M$ of $K_v$; in the present case of hyperelliptic
    curves this is possible using factorisation of univariate polynomials over $M$
    as described in \cite[\S5.3]{mueller:computing}.

    So it remains to discuss how to represent Zariski closures of prime divisors on
    $C$ on the affine piece
    \[
    \mathcal{C}^a=\Spec\left(\cO[x,y]/(y^2-f(x))\right).
    \]
      If $D_1=\sum^d_{i=1} (P_i)$ is reduced, then we can use
    a representing ideal
    \[
      (a(x), y-b(x)),
    \]
    where $a(x)\in \cO[x]$ has roots $x(P_1),\ldots,x(P_d)$ and does not vanish
    modulo $\pi_v$ and $b(x)\in\cO[x]$ does not vanish modulo $\pi_v$ and
    satisfies $y(P_i)=b(x(P_i))$ for $i\in\{1,\ldots,d\}$.
    This is commonly referred to as Mumford representation, see
    \cite[3.19]{mumford:tata1}.
    In particular, if $D_1=(P_1)$, where $P_1\in C(K_v)$, then we can take the ideal
    \[
     (x-x(P_1), y-y(P_1)).
    \]
    The other case we have to consider is the case $D_1=(P_1)+(P^-_1)$, where $P_1$
    is defined over an extension of $K_v$ of degree at most~2 and $P^-_1$ is the
    image of $P_1$ under the hyperelliptic involution.
    Then we can simply use the ideal
    \[
     (x-x(P_1)).
    \]

    For Step (3) we refer to \cite[\S4.5]{mueller:computing}.
    In brief, we first compute the intersection matrix $M$ of $\mathcal{C}_v$ and its
    Moore-Penrose pseudoinverse $M^+$.
    Suppose that the special fiber $\mathcal{C}_v$ is given by $\sum^m_{i=0}n_i\Gamma_i$,
    where $\Gamma_0,\ldots,\Gamma_m$ are the irreducible components of $\mathcal{C}_v$.
    We also need the vectors $s(D)$ and $s(E)$ of intersection multiplicities, where
    \[
    s(D)=\left(n_0\cdot i_v(\overline{D},\Gamma_0),\ldots,n_m\cdot i_v(\overline{D},\Gamma_m)\right)^T;
    \]
    and $s(E)$ is defined similarly.
    These can be computed using the techniques introduced in Step (2) above.
    Then we have
    \[
      i_v(\Phi(D),\overline{E})=-s(E)^T\cdot M^+\cdot s(D).
    \]

    We have not discussed how we can compute a finite set $U$ of places of $K$ such
    that we have $h_v(D,E)=0$ for all $v\notin U$.
    This is discussed in \cite[\S4.2, \S5.2]{mueller:computing}.
    Here we only mention that it suffices to compute $U$ containing all bad places (that is, all places
    $v$ such that $\mathrm{ord}_v(2\cdot\mathrm{disc}(f))>0$) and all places $v$ such that $D$
    and $E$ have nontrivial common support modulo $\pi_v$.
    The latter can be computed as follows, where $D=D_1-D_2$ and $E=E_1-E_2$ are as
    above.

    We only discuss the computation of all $v$ such that $D_1$ and $E_1$ have
    nontrivial common support modulo $\pi_v$.
    Let $I_{D_1}$ and $I_{E_1}$ denote representing ideals of the Zariski closures of $D_1$
    and $E_1$ on the affine piece
    \[
      \Spec(\cO_K[x,y]/(y^2-f(x)))
    \]
    of the Zariski closure of $X$ over $\Spec(\cO_K)$.
    We assume that $\cO_K$ is Euclidean; the general case can be reduced to this
    situation using a straightforward trick discussed in
    \cite[\S4.2]{mueller:computing}.
    If $B$ is a Gr\"obner basis of the ideal
    \[
      (y^2-f(x)) + I_{D_1} + I_{E_1},
    \]
    over $\cO_K$, then $B$ contains a unique element $q\in\cO_K$ (cf.
    \cite[Lemma~4.3]{mueller:computing}).
    Factoring $(q)$ yields a set of places containing all places such that $D_1$
    and $E_1$ have nontrivial common support on the reduction of the affine piece
    given by $y^2=f(x)$.
    Repeating this process for the other standard affine piece $y^2=x^{2g-2}f(1/x)$
    yields all places $v$ such that $D_1$ and $E_1$ intersect nontrivially modulo
    $\pi_v$.

    \subsection{Computing $p$-adic heights above $p$}\label{sec:local_above_p}

    We now describe the local contribution at a place $v\mid{}p$.

    \begin{definition} Let $D,E \in \Div^0(C)$ have disjoint support. The local height pairing at a place $v\mid{}p$ is given by the following Coleman integral:
      \begin{equation*}
        h_v({D},{E}):=t_v\left(\int_{E}\omega_{D}\right),
      \end{equation*} where $t_v$ is the trace map
      determined by the decomposition of $\ell_v$ (see~\eqref{tdefnd}), and $\omega_{D}$ is a differential associated to $D$.
    \end{definition}

    We start by reviewing the construction of  $\omega_{D}$.   Let $T(K_v)$ denote the subgroup of differentials on $C$ of the third
    kind. We are interested in a particular
    subgroup of $T(K_v)$ whose elements are the logarithmic differentials,
    i.e., those of the form $\frac{df}{f}$ for $f \in K_v(C)^*$. We denote
    this subgroup as $T_l(K_v)$.  Letting $\hdr^{1,0}(C/K_v)$ denote the space of holomorphic differentials and $A = \Pic^0(C)$, we have the short exact sequence
    $$ 0 \lra \hdr^{1,0}(C/K_v) \lra T(K_v)/T_l(K_v) \lra A(K_v) \lra 0.$$

    This sequence has a natural identification with the $K_v$-rational points of an exact sequence of commutative algebraic groups over $K_v$:
    $$0 \lra  \hdr^{1,0}(C/K_v) \lra \mathcal{U} \lra A \lra 0,$$ where $\mathcal{U}$ is
    the universal extension of $A$ by a vector group and $\hdr^{1,0}(C/K_v)
    \cong \Ga^g.$

    Now as $K_v$ is $p$-adic, we will make use of the fact that we have a
    logarithmic homomorphism defined on an open subgroup of the points of
    any commutative $p$-adic Lie group, $G$, to the points of its Lie
    algebra $\Lie(G)$. When $G = \mathcal{U}$ or $A$, the open subgroup on which the
    logarithm converges has finite index, so the homomorphism can be
    uniquely extended to the entire group. We denote this extension as
    $\log_{\mathcal{U}}$ or $\log_A$, respectively. Since the logarithm is functorial
    and equal to the identity on $\hdr^{1,0}(C/K_v) $, we have the
    following:

    \begin{proposition}\label{comm}There is a canonical
      homomorphism $$\Psi: T(K_v)/T_l(K_v) \lra \hdr^1(C/K_v)$$ which is the
      identity on differentials of the first kind and makes the following
      diagram commute:
      $$\xymatrix{
        0 \ar[r] & \hdr^{1,0}(C/K_v) \ar[r]\ar@{=}[d] & \mathcal{U}(K_v) \ar[d]^{\Psi = \log_{\mathcal{U}}}
        \ar[r] & A(K_v) \ar[r]\ar[d]^{\log_A}  & 0  \\
        0 \ar[r] & \hdr^{1,0}(C/K_v) \ar[r] &\hdr^1(C/K_v) \ar[r]
        &H^1(C,\cO_{C/K_v}) \ar[r] &0.}$$
    \end{proposition}

    Note that the map $\Psi$ takes a differential of the third kind on $C$
    to a differential of the second kind modulo exact differentials,
    sending log differentials to $0$. It
    can be extended to a linear map from the $K_v$-vector space of all
    differentials on $C/K_v$ to $\hdr^1(C/K_v)$ by writing an arbitrary
    differential $\nu$ as a linear combination $\nu = \sum\alpha_i\mu_i +
    \gamma$, where $\mu_i$ is of the third kind, $\alpha_i \in
    \overline{K_v}$, and $\gamma$ is of the second kind on $C$. We then
    define $\Psi(\nu) = \sum \alpha_i \Psi(\mu_i) + [\gamma]$.

    Now recall that we have at our disposal the complementary subspace
    $W=W_v$. It allows us to isolate a canonical form $\omega_D$ with
    residue divisor $D$ as follows:
    \begin{definition}\label{omd}
      Let $D \in \Div^0(C)$. Then $\omega_{D}$ is the unique form of the third kind
      satisfying $$\res(\omega_{D})={D}, \quad \lag(\omega_{D})\in W.$$
    \end{definition}

  \subsection{Computing the global $p$-adic height}\label{sec:height_algo}
  Using the material in Sections~\ref{sec:local_away_from_p}  and~\ref{sec:local_above_p}, we may now give an algorithm to compute global $p$-adic heights on Jacobians of hyperelliptic curves.

\begin{algorithm}[Global $p$-adic height pairing]\label{algo:height22}
    $\,$\\
    Input: \begin{itemize}\item Genus $g$ hyperelliptic curve $X$ over $\Q$ of the form $y^2 = f(x)$, with $f$ separable, $\deg f = 2g+1$,
    \item Prime $p$ of good ordinary reduction for $X$,  \item  Divisors $D,E \in
        \Div^0(X)$ with disjoint support.  \end{itemize}
    Output: \begin{itemize}
    \item Global $p$-adic height pairing $h(D,E) = \sum_{v} h_v(D,E)$ \end{itemize}
    Algorithm:
    \begin{enumerate}
    \item \textbf{Heights away from $p$}
    \begin{enumerate}
    \item \textbf{Find relevant places}. Compute the following set $U$ of non-archimedean places of $K$:
\[
    U=\{v : v\mid{}2\disc(f)\}\cup\{v:\mathrm{supp}(D\bmod v)\cap\mathrm{supp}(E\bmod
v)\ne\emptyset\}.
\]  \item \textbf{Local computations}. For each $v\in U$ go through the following steps.
\begin{enumerate}
    \item \textbf{Regular models}. Compute a desingularization $\mathcal{C}$ of the
Zariski closure of $X\times_K K_v$ over $\cO=\cO_v$ in the strong sense.
\item \textbf{Ideal representatives}. Write $D=D_1-D_2$ and $E=E_1-E_2$, where $D_i,\,E_j$
are effective and find representatives $I_{\overline{D_i}}$ and $I_{\overline{E_j}}$ of
the Zariski closures $\overline{D_i},\overline{E_j}$ of $D_i,\,E_j$, respectively, on an affine piece of $\mathcal{C}$ containing
$\mathrm{supp}(D)\cap\mathrm{supp}(E)$.
\item \textbf{Horizontal data}. Compute $i_v(\overline{D_i},\overline{E_j})$ for
$i,j\in\{1,2\}$ using~\eqref{IntForm}.
\item \textbf{Vertical data}. Compute the Moore-Penrose pseudoinverse $M^+$ of
the intersection matrix of the special fiber $\mathcal{C}_v=\sum^m_{i=0}n_i\Gamma_i$ and
for $H\in\{D,E\}$, the vectors
    \[s(H)=\left(n_0\cdot i_v(\overline{H},\Gamma_0),\ldots,n_m\cdot
i_v(\overline{H},\Gamma_m)\right)^T.\]
\item \textbf{Local height pairing}. Let $k_v$ be the residue field at $v$ and
set
\[
h_v(D,E)=\left(s(E)\cdot M^+\cdot
s(D)^T-\sum_{i,j} i_v(\overline{D_i},\overline{E_j})\right)\cdot \log(\# k_v).
\]

\end{enumerate}
\item \textbf{Global height pairing away from $p$}. Compute $\sum_{v\nmid p} h_v(D,E)$.
\end{enumerate}

\item \textbf{Height above $p$} \begin{enumerate}
    \item \textbf{From $D$ to $\omega_D$}. Choose $\omega$ a differential of the third kind with $\Res(\omega) = D$ and compute $\log(\omega) = \Psi(\omega)$ for $\omega$. Using the decomposition $$H_{dR}^1(C/K_v) \simeq H_{dR}^{1,0}(C/K_v)\oplus W,$$ write $$\log(\omega) = \eta + \log(\omega_{D}),$$
        where $\eta$ is holomorphic, and $\log(\omega_{D}) \in W$. Now, using this holomorphic component $\eta$, we have $$\omega_{D} := \omega - \eta.$$

    \item \textbf{Coleman integration}
    \begin{enumerate}\item ...of a holomorphic differential. Compute  $\int_{E}\eta$, as in \cite{bbk}.

    \item {...of a meromorphic differential.} Let $\phi$ be a $p$-power lift of Frobenius and set $\alpha := \phi^*\omega - p\omega$. Write $E = \sum_i E_i$, where $E_i = (R_i) - (S_i)$ for points $R_i,S_i$ on $X$. Then for $\beta_i$ a differential with residue divisor $E_i$, we compute

     \begin{align*}\hspace{1in}\int_{E}\omega &= \sum_i \int_{E_i} \omega\\
    &=\sum_i \frac{1}{1-p}\left(\Psi(\alpha) \cup \Psi(\beta_i) + \sum \Res \left(\alpha\int\beta_i\right) - \int_{\phi(S_i)}^{S_i} \omega - \int_{R_i}^{\phi(R_i)}\omega\right),\end{align*} as described in \cite{bb}.
\end{enumerate}
    \item\label{Step5} \textbf{Height pairing above $p$}. Subtract the integrals to recover the pairing at $p$: $$h_p(D,E) = \int_E \omega_D = \int_E \omega -\int_E \eta.$$
    \end{enumerate}
    \item \textbf{Global $p$-adic height pairing}. Return the sum of  1(c) and 2(c). \end{enumerate}
    \end{algorithm}

\begin{remark} Note that our current implementation of Algorithm~\ref{algo:height22} further assumes, in Step 2(b)(ii), that $R_i,S_i \in C(\Q_p)$.\end{remark}

    \subsection{Computing the $p$-adic regulator}\label{sec:comp_reg}
    In this section, we explain how we use Algorithm~\ref{algo:height22}, in practice, to compute the $p$-adic regulator of $A/\Q$, where $A$
    is a Jacobian surface associated to a curve $X/\Q$ of genus~2   and $p$ is a prime such that $A$
    has good ordinary
    reduction at $p$.

    Suppose that $P,Q\in A(\Q)$ are distinct and that
    we want to compute the $p$-adic height pairing of $P$ and $Q$.
    If we can find representatives $D_1=(P_1)-(P_2)$ and $D_2=(Q_1)-(Q_2)$ of $P$ and $Q$,
    respectively, where $P_1,P_2,Q_1,Q_2\in X(\Q)$ are all distinct, then we can simply apply
    Algorithm~\ref{algo:height22} directly to compute $h_p(D_1,D_2)$.

    However, in many situations, it is not possible to find
    representatives of $P$ and $Q$ whose support consists of $\Q$-rational points.
    We can still compute the $p$-adic height pairing if we can find representatives
    $D_1,D_2\in \Div^0(X)(\Q)$ of $P$ and $Q$, respectively, such that
    \[D_1\times_\Q\Q_p= (P_1)+(P_2)-(R)-(R^-),\]
    \[D_2 \times_\Q\Q_p= (Q_1)+(Q_2)-(S)-(S^-),\] where
    $P_1,P_2,Q_1,Q_2,R,S\in C(\Q_p)$ are pairwise distinct.
    It is explained in Section~\ref{sec:local_away_from_p} how to find ideal
    representations of the positive and negative parts of $D_1$ and $D_2$, respectively, which is all we need to
    compute the local height pairings away from $p$.
    The latter is very similar to the computation of canonical real-valued height
    pairings on Jacobians of hyperelliptic curves as discussed in
    \cite{mueller:computing} and has been implemented in {\tt Magma}.

    If we want to compute the $p$-adic height pairing of $P$ with itself,
    then we simply compute minus the $p$-adic height pairing of $P$ with $-P$ using
    the techniques discussed above.

    \begin{remark}\label{rk:general_hyper_rational}
      In principle this algorithm can be generalized immediately to
      hyperelliptic curves over $\Q$ of arbitrary genus, with minor
      subtleties if the genus is odd and the degree is even. See the
      discussion in \cite{mueller:computing}.  However, to find the
      $p$-adic regulator precisely (and not only up to a rational square), we first need a set of generators for
      $A(\Q)/A(\Q)_{\textrm{tors}}$. Given generators of a finite index subgroup of
$A(\Q)/A(\Q)_{\textrm{tors}}$, a general algorithm for this
      computation is presented by
      Stoll in \cite{stoll:height_constantII}, but currently it is
      only feasible if the rank is zero or $X$ has genus~2 (using
      current work in progress due to Stoll this can be extended to
      genus~3).  Hence we can only hope to compute the $p$-adic
      regulator up to a rational square in general.
    \end{remark}

    \begin{remark}
    It is useful to note that in order to compute the $p$-adic regulator in the
    genus~2 case, we need not work with generators of the free
    part of the Mordell-Weil group itself.
    Indeed, if we have a set of points generating a subgroup $G$ of finite index,
    then we can easily compute the index of $G$ by computing and comparing the real-valued
    regulators of $G$ and of $A(\Q)/A(\Q)_{\textrm{tors}}$, see Remark~\ref{rk:general_hyper_rational}.
    Since the $p$-adic height is quadratic, it suffices to compute the $p$-adic regulator
    of $G$ in order to deduce the $p$-adic regulator.
    This helps in finding points with representatives of the required form.
    \end{remark}

    \begin{remark}
    Suppose that $X$ is hyperelliptic and  defined over a number field $K$.
    Then we can still compute the $p$-adic regulator exactly as above if all completions
    $K_v$ at places $v\mid{} p$ satisfy $K_v\cong \Q_p$.
\end{remark}

Finally, following \cite{stein-wuthrich:shark}, we note that the $p$-adic regulator has a natural normalization from Iwasawa theory, coming from the choice of topological generator $\gamma$.  This is done so that the global $p$-adic height depends only on the choice of isomorphism $\Gamma \ra \Z_p$, instead of on the $\Z_p$-extension. This normalization is carried out by dividing $h(P_i,P_j)$ by $\log_p(\kappa(\gamma))$, or, alternatively, since the $p$-adic regulator involves a basis of dimension $r$, by taking
$$\Reg_{\gamma}(A/\Q) = \frac{\Reg_p(A/\Q)}{(\log_p(\kappa(\gamma)))^r}.$$

\section{Evidence for rank 2 Jacobians of genus 2 curves}\label{sec:rank2}
As we now have algorithms to compute the $p$-adic regulator and $p$-adic
$L$-series, we proceed to verify Conjecture~\ref{pbsd} for specific abelian
varieties, using BSD data from \cite[Table 2]{empirical}.
We take as our list of candidate modular abelian varieties $A$ those appearing
in \cite{empirical} of rank 2.

We were able to find generators represented by divisors whose support consists only of
$\Q$-rational points
for each of the rank~2 Jacobians taken from \cite{empirical} except for the one with level~167.
Hence we used the easier first approach outlined in Section~\ref{sec:comp_reg}
to compute the $p$-adic regulator for all Jacobians except for the one associated to level~167.

In order to compute the special values $\cL^*_p(A, 0)$, we used an implementation in Sage of the
algorithm outlined in Section~\ref{sec:OMS}.
Previously, we had used Algorithm~\ref{algo:height22} obtaining fewer
digits of precision.
The results agreed up to the precision obtained using the latter.

Table~\ref{t:1} is taken from \cite{empirical} and contains minimal models (in the sense
of \cite{liu:conducteur}) for each curve from \cite{empirical} whose Jacobian
    variety has Mordell-Weil rank 2 over $\Q$, as well as the corresponding level $N$.

The data presented in this section proves the following:

\begin{theorem}
Assume that for the Jacobians of all curves in Table~\ref{t:1} the Shafarevich-Tate group
over $\Q$ is~2-torsion.
Then Conjecture~\ref{pbsd} is satisfied up to the respective precision specified in
the tables below for the Jacobians of all curves in Table~\ref{t:1} at all good ordinary $p < 100$ satisfying the hypotheses of our algorithms.
\end{theorem}
\begin{remark}
The assertion that all Shafarevich-Tate groups are 2-torsion for these abelian varieties follows from the classical
conjecture of Birch and Swinnerton-Dyer by \cite{empirical}.
\end{remark}

    For our algorithms, we take the integral models in \cite[Table 1]{empirical} and do a change of
coordinates to obtain the corresponding models of the form $y^2 = f(x)$. We record both models in Table~\ref{t:1}
      \begin{center}
    \begin{table}[h!]
        \begin{tabular}{| r|l | l| }
        \hline
        $N$ & $(g(x),h(x))$ for integral model  & $f(x)$ for $y^2 = f(x)$ model\\
        \hline
        67 &  $(x^5-x, x^3+x+1)$ & $x^{6} + 4 x^{5} + 2 x^{4} + 2 x^{3} + x^{2} - 2 x + 1$\\
        73 & $( -x^5 - 2x^3 + x, x^3 + x^2 + 1)$ & $ x^{6} - 2 x^{5} + x^{4} - 6 x^{3} + 2 x^{2} + 4 x + 1$ \\
        85 & $(x^4 + x^3 + 3x^2-2x+1,x^3+x^2+x)$ & $x^{6} + 2 x^{5} + 7 x^{4} + 6 x^{3} + 13 x^{2} - 8 x + 4$\\
        93 & $(-2x^5 + x^4 + x^3,x^3 +x^2 + 1)$& $x^{6} - 6 x^{5} + 5 x^{4} + 6 x^{3} + 2 x^{2} + 1$\\
        103 & $(x^5 +x^4,x^3 + x^2 + 1)$ & $x^{6} + 6 x^{5} + 5 x^{4} + 2 x^{3} + 2 x^{2} + 1$\\
        107 & $(x^4 - x^2 - x - 1,x^3 + x^2 + 1)$ & $x^{6} + 2 x^{5} + 5 x^{4} + 2 x^{3} - 2 x^{2} - 4 x - 3$\\
        115 & $(2x^3 + x^2 + x,x^3 + x + 1)$ & $x^{6} + 2 x^{4} + 10 x^{3} + 5 x^{2} + 6 x + 1$\\
        125,A & $(x^5 + 2x^4 + 2x^3 +x^2 - x -1,x^3 + x + 1)$ & $x^{6} + 4 x^{5} + 10 x^{4} + 10 x^{3} + 5 x^{2} - 2 x - 3$\\
        133,B & $(-x^5 + x^4 -2x^3 +2x^2-2x,x^3 + x^2 + 1)$ & $x^{6} - 2 x^{5} + 5 x^{4} - 6 x^{3} + 10 x^{2} - 8 x + 1$\\
        147 & $(x^5 + 2x^4 + x^3 + x^2 + 1,x^3 + x^2 + x)$ & $x^{6} + 6 x^{5} + 11 x^{4} + 6 x^{3} + 5 x^{2} + 4$\\
        161 & $(x^3 + 4x^2 + 4x + 1,x^3 + x + 1)$ & $x^{6} + 2 x^{4} + 6 x^{3} + 17 x^{2} + 18 x + 5$\\
        165 & $(x^5 + 2x^4 + 3x^3 + x^2-3x,x^3 + x^2 + x)$ & $x^5 + 5x^4 - 168x^3 + 1584x^2 - 10368x + 20736$\\
        167 & $(-x^5 - x^3 - x^2 -1,x^3 + x + 1)$  & $x^{6} - 4 x^{5} + 2 x^{4} - 2 x^{3} - 3 x^{2} + 2 x - 3$\\
        177 & $(x^5 + x^4 + x^3,x^3 + x^2+1)$ & $x^{6} + 6 x^{5} + 5 x^{4} + 6 x^{3} + 2 x^{2} + 1$\\
        188 & $(x^5 - x^4 + x^3 + x^2 - 2x+1,0)$ & $x^5 - x^4 + x^3 + x^2 - 2x+1$\\
        191 & $(-x^3 + x^2 +x,x^3 + x+ 1)$ & $x^{6} + 2 x^{4} - 2 x^{3} + 5 x^{2} + 6 x + 1$\\
            \hline
        \end{tabular}

     \caption{$\;\;$Levels, integral models $y^2  + h(x)y = g(x)$, simplified models $y^2 = f(x)$}\label{t:1}
    \end{table}
      \end{center}
Let us recall what is known about computing the quantities appearing on the right
side of Equation~\eqref{eq-pbsd} which we have not addressed so far.
As described in \cite{empirical}, the order of the torsion subgroups and the Tamagawa
numbers are computable.
For the Jacobians of the curves in Table~\ref{t:1}, we list these values, taken from
\cite[Table~2]{empirical}, in Table~\ref{t:3}.
While no general algorithm has yet been developed and implemented
to compute the order of the Shafarevich-Tate group
$\Sha(A/\Q)$
for the Jacobians of
each of the curves in Table~\ref{t:1}, the conjectural order $\Sha?$ of the group is also
given, conditional on the classical BSD conjecture~\ref{bsdabvar} (and equal to the order of
$\Sha(A/\Q)[2]$).

\begin{remark}
There is a general approach to computing $\Sha(A/\Q)$ for the rank $2$
Jacobians in Table~\ref{t:1}, which is to use Heegner points and
Kolyvagin's Euler system to give an explicit upper bound, then compute
the remaining Selmer groups.  It would be an interesting project
to systematically develop this approach, by generalizing
\cite{bsdalg1, miller:bsd1} to this new setting.
\end{remark}

    \begin{table}[h!]
      \begin{center}
        \begin{tabular}{| r| c | c | c |}
        \hline
        $N$ & $c_v$'s & $| A(\Q)_{\textrm{tors}}|$ & $\Sha?$  \\
        \hline
        67 & 1 & 1 &1 \\
        73 & 1 & 1 & 1\\
        85 & 4,2 & 2 & 1\\
        93 & 4,1 & 1 & 1\\
        103 & 1 & 1 & 1 \\
        107 & 1 & 1 & 1 \\
        115 & 4,1 & 1 & 1\\
        125,A & 1 & 1& 1\\
        133,B & 1,1 & 1 & 1 \\
        147 & 2,2 & 2 & 1 \\
        161 & 4,1 & 1 & 1 \\
        165 & 4,2,2 & 4 & 1\\
        167 & 1 & 1 & 1\\
        177 & 1,1 & 1 & 1\\
        188 & 9,1 & 1 & 1\\
        191 & 1 & 1 & 1 \\
            \hline
        \end{tabular}
      \end{center}
      \caption{$\;\;$BSD data for rank 2 Jacobians of genus 2 curves}\label{t:3}
    \end{table}

     Table~\ref{tabglob} below  provides the local height pairings away from $p$ for $N\ne167$.
    The global generators for $A(\Q)/A(\Q)_{\tors}$ that we used are given as divisor classes $[P-Q],[R-S]$, where $P,\,Q,\,R,\,S\in C(\Q)$.
    Points at infinity are denoted by $\infty_a$, where $a$ is equal to $y/x^3$ evaluated
    at $\infty_a$.
    The heights list has three entries giving the nontrivial local height pairings
 $h_v((P)-(Q),(R)-(S)),\,h_v((P)-(Q),(-Q)-(-P))$
    and $h_v((R)-(S),(-S)-(-R))$ for $v\ne p$.
 For two divisors $D$ and $E$, this data is returned as a list of pairs $[v,d_v]$, where
 $v$ is a prime and $h_v(D,E)=d_v\cdot\log_p(v)$.
    \begin{table}[h!]
     \begin{center}
     \begin{tabularx}{\textwidth}{|r| c|X|}
     \hline
     $N$ & global generators for $A(\Q)/A(\Q)_{\tors}$ & heights $[[v, h_v]]$ \\
     \hline
     67 &  $[( -1, 0 )-\infty_{-1}], [( 0, -1 )- \infty_0]$ &  $[\;], [\;], [\;]$ \\
     73 &  $[( -1, -2 )-\infty_{-1}],[( 0, -1 )-\infty_0 ]$ & $[\;], [[ 3, 1 ]], [\;]$\\
     85 & $[(-1, -2 )-\infty_{-1}],[( 1, -4 )-\infty_0 ]$ & $[[ 2, -1 ]],[[ 5, \frac{1}{2}
]],[[ 5, \frac{1}{2}]]$\\
93 & $[( -1, -2 )-\infty_{-1}],[( 1, -3 )-\infty_0 ]$ & $[\;], [[ 3, \frac{1}{2} ]],[[
3, \frac{1}{2} ]]$ \\
103 & $[( -1, -1 )- \infty_{-1}],[( 0, -1 )-\infty_0 ]$& $[\;], [\;], [\;] $\\
107 & $[( -1, -1 )-\infty_{-1}],[( 1, -2 )-\infty_0 ]$& $[\;], [\;], [\;]  $ \\
115 & $[( 1, -4 )-\infty_{-1}], [( -2, 2 )-\infty_0 ]$&  $[[ 3, -1 ]],[[ 5,
\frac{1}{2} ]],[[ 5, \frac{1}{2} ]]$\\
125,A & $[( -1, 0 )-\infty_{-1}],[( 1, -4 )-\infty_0 ]$ & $[[ 2, -1 ]],[\;],[[ 5, 1 ]]$ \\
133,B & $[( 0, -1 )-\infty_{-1}],[( 1, -2 )-\infty_0 ]$& $[\;], [\;], [\;]$ \\
  147 & $[( -1, -1 )-\infty_{-1}],[( -3, 7 )-\infty_0 ]$ & $[[ 2, -1 ]],[[ 3,
\frac{1}{2} ]],[[ 7, \frac{1}{2}]]$\\
161 & $[( 1, -5 )-\infty_{-1}],[( \frac{2}{3}, -3 )-\infty_0 ]$ & $[\;],[[ 7,
\frac{1}{2} ]],[[ 5, 1 ],[ 7, \frac{1}{2} ]]$ \\
        165 & $[( -8, -528 )- ( 0, -144)], [( 8, 80 )-( 0, 144)]$ & $[[ 2, 2 ],[ 3, -\frac{1}{2} ]]$,
        $[[ 2, -2 ],[ 11, \frac{1}{2} ],[ 3, \frac{3}{2}]],$\\
      &&  $[[ 2, -2 ],[ 5, \frac{1}{2} ],[ 3, \frac{1}{2} ]]$\\
  177 & $[( 0, 0 )-\infty_{-1}],[( -\frac{2}{3}, -\frac{7}{27} )-\infty_0 ]$ & $[[ 3, 1
]],[\;],[[ 3, -2 ],[17, 1 ]]$\\
188 & $[(0,1)- \infty_{-1}],[(-1,-1)-(2,5)]$ &  $[[2,1]],[[2,\frac{2}{3}]],
[[2,\frac{2}{3}],[5,1]]$\\
191 & $[( 0, -1 )-\infty_{-1}],[( -2, 10 )-\infty_0 ]$ & $[\;],[\;],[[ 11, 1 ]]$ \\
        \hline
            \end{tabularx}
      \end{center}
      \caption{$\;\;$Global generators and intersection data}\label{tabglob}
    \end{table}
    \begin{remark}The generators given for $N=125, A$ are actually generators for an index
2 subgroup of $A(\Q)/A(\Q)_{\tors}$, since an actual set of generators for the full group $A(\Q)/A(\Q)_{\tors}$ whose support solely consisted of non-Weierstrass points was not readily available.
For $N=167$, we had to use generators (of finite index subgroups) represented by divisors with pointwise
$\Q_p$-rational support; see Section~\ref{sec:comp_reg}.
\end{remark}

    For the computation of the special values $\cL_p(A,0)$, we need the normalization factor
    $\delta^+$, so we have to find a fundamental
    discriminant $D>0$ such that for some good ordinary $p_0$ we have $\gcd(p_0N, D)=1$
and the analytic rank of $A_\psi$ is zero, where $\psi$ is the quadratic character
associated to $\Q(\sqrt{D})$. See Subsection~\ref{sec:norm}.
    The real period $\Omega^+_A$ for our Jacobians can be found in
    \cite{empirical}; there it was
    computed using the observation that $(\omega_1,\omega_2)$ as in
Remark~\ref{rk:find_eta} is a basis of integral 1-forms
    for all abelian varieties we consider.
    It is not difficult to show that the corresponding fact also holds for all $A_\psi$ and
    hence, using Remark~\ref{rk:find_eta},  we found that $\eta_\psi=1$ in all~16 cases.
    We list, for each level, the quantities needed to find $\delta^+$ in Table~\ref{t:D}.

\begin{table}[ht]

\centering
        \begin{tabular}{| r|c|c|c|c | c |c| c | c |c|}
              \hline
        $N$ & $D$ & $\eta_\psi$&$p_0$&$[0]^+_{A_\psi}$ &
$\frac{\eta_\psi\cdot
    L(A_\psi,1)}{D\cdot\Omega^+_A}$ &$\delta^+$&$c_v(A_\psi)'s$&$| A_\psi(\Q)_{\textrm{tors}}|$
& $\Sha_\psi?$  \\
        \hline
        67 &5&1&19 &$16$ & 4&$1/4$&    1,1 & 1 &4 \\
        73 & 5&1&11&16 & 4 &$1/4$ &  1,1 & 1 & 4\\
        85 & 61&1&41&64 & 16&$1/4$ &   4,2,8 & 2 & 1\\
        93 &5&1&11 &$-16$ & 4&$-1/4$&    4,1,1 & 1 & 1\\
       103 &5&1&11 &16 & 4&$1/4$  & 1,1 & 1 & 4 \\
   107 &5&1&19 &16 & 4&$1/4$    &1,1 & 1 & 4 \\
   115 & 89&1&11 &64 &16&$1/4$  &   4,1,4 & 1 & 1\\
 125,A &17&1&19 & 16&4 &$1/4$   & 1,1 & 1& 4\\
 133,B & 5&1&29&$-16$ &4 &$-1/4$ &   1,1,1 & 1 & 4 \\
   147 & 5&1&31 &$-16$ & 4&$-1/4$&    2,2,2 & 2 & 2 \\
   161 &53&1&11 & 64& 16&$1/4$  &  4,1,4 & 1 & 1 \\
   165 &89&1 &17&64 & 16&$1/4$  &  4,2,2,16 & 4 & 1\\
   167 & 5&1&31&$-16$ & 4 &$-1/4$ &  1,1 & 1 & 4\\
   177 & 5&1&19&16 &  4&$1/4$  & 1,1,1 & 1 & 4\\
   188 & 233&1&19&144 &36 &$1/4$ &   9,1,4 & 1 & 1\\
   191 &33&1&31 & 16& 4&$1/4$  &  1,1 & 1 & 4 \\
        \hline
    \end{tabular}
\vspace{1ex}
              \caption{Rank zero twist data}\label{t:D}
 \end{table}

For good measure, we also verified:
\begin{proposition}
 The classical (and hence, for all primes $p$ of good ordinary reduction, the $p$-adic) Birch and Swinnerton-Dyer
conjecture holds for all~16 twists $A_\psi$ in Table~\ref{t:D} under the assumption that
$\Sha(A_\psi/\Q)$ is 2-torsion.
\end{proposition}
\begin{proof}
See Table~\ref{t:D}, noting that $\Sha_\psi?$ is equal to $|\Sha(A_\psi/\Q)[2]|$.
\end{proof}
\begin{remark}
For the computation of the Tamagawa numbers we used
\cite[Theorem~1.17]{bosch-liu:rat_pts}.
Suppose that $v\mid{}D$, but that $v$  does not divide the conductor of $A$.
Then the
twisted curve $X_\psi$ has bad reduction at $v$, but acquires good reduction
over a quadratic extension.
The classification of Namikawa and Ueno \cite{namikawa-ueno}
shows that in this case $X_\psi$ must have reduction type $[I^*_{0-0-0}]$ at $v$.
Since the geometric component group of the N\'eron model is isomorphic to
$(\Z/2\Z)^4$ for this reduction type, we always have $c_v(A_\psi)\mid 16$.
\end{remark}

The tables below show the specific primes $p$ and precision $O(p^n)$ for each level
$N$ for which we have tested Conjecture~\ref{pbsd}.

\begin{remark}\label{rem:mod}A note on our models and choices of primes. Since our
$p$-adic heights algorithm requires that the curve be given by an odd degree model, for
each curve $y^2 = g(x)$, we consider those good ordinary primes $p$ for which $g(x)$ has a $\Q_p$-rational zero and do another change of coordinates to obtain the odd model $y^2 = f(x)$, with $f(x) \in \Q_p[x]$. We compute the $p$-adic regulators and $p$-adic $L$-values for these primes.\end{remark}

\clearpage

\subsection{\textbf{$N=67$}} We have the following data:\\

\begin{center}
    \begin{tabular}{| l | l | l |}
    \hline
$p$-adic regulator $\Reg_p(A/\Q)$ & $p$-adic $L$-value & $p$-adic multiplier $\epsilon_p(A)$ \\
\hline
$ 905422 + O(7^{8}) $ & $ 4616447 + O(7^{8}) $ & $ 953283 + O(7^{8}) $\\
$ 655636176 + O(13^{8}) $ & $ 3718847 + O(13^{8})$ & $121846702 + O(13^{8}) $\\
$ 6411910349 + O(17^{8}) $ & $ 490126740 + O(17^{8})$ & $ 2996208382 + O(17^{8}) $\\
$ 1955457580 + O(19^{8}) $ & $ 205789013 + O(19^{8})$ & $6722090086 + O(19^{8}) $\\
$ 6490501114813 + O(37^{9}) $ & $ 1520740814200 + O(37^{8})$ &$ 1763856795912 + O(37^{8}) $\\
$ 119112862323467 + O(41^{9}) $ & $ 6899026979535 + O(41^{8})$ & $ 604530321123 + O(41^{8}) $\\
$ 231768637543452 + O(43^{9}) $ & $ 11662319738050 + O(43^{8})$ & $ 11664993765232 + O(43^{8}) $\\
$ 258343847102710 + O(47^{9}) $ & $ 6617527122585 + O(47^{8}) $ & $ 21577206386081 + O(47^{8}) $\\
$ 7291679100956850 + O(59^{9}) $ & $ 72703739307529 + O(59^{8}) $& $ 77184936742982 + O(59^{8}) $\\
$ 6048812062982476 + O(61^{9}) $ & $ 174305066216353 + O(61^{8}) $& $ 133406272889885 + O(61^{8}) $\\
$ 53277934412195075 + O(73^{9}) $ & $ 552479201354189 + O(73^{8})$ & $ 460739420635942 + O(73^{8}) $\\
$ 9278983589215557 + O(79^{9}) $ & $ 88027589402068 + O(79^{8}) $ & $ 801037408797804 + O(79^{8}) $\\
$ 157708559779041510 + O(83^{9}) $ & $ 1578704504708054 + O(83^{8})$ & $ 162512920516158 + O(83^{8}) $\\
\hline
 \end{tabular}
 \end{center}

 \vspace{.5in}

\subsection{\textbf{$N=73$}} We have the following data:\\

\begin{center}
    \begin{tabular}{| l | l | l |}
    \hline
$p$-adic regulator $\Reg_p(A/\Q)$ & $p$-adic $L$-value & $p$-adic multiplier $\epsilon_p(A)$ \\
\hline
$ 163731997 + O(11^{8}) $ & $ 183868925 + O(11^{8}) $ & $
192773925 + O(11^{8}) $\\
$ 482988818 + O(13^{8}) $ & $ 522787644 + O(13^{8}) $ & $
757562196 + O(13^{8}) $\\
$ 51174691892 + O(23^{8}) $ & $ 46581832325 + O(23^{8}) $ & $
48224542827 + O(23^{8}) $\\
$ 553299007790 + O(31^{8}) $ & $ 440555494391 + O(31^{8}) $ & $
850258335981 + O(31^{8}) $\\
$ 5421948177967 + O(41^{8}) $ & $ 5077531013725 + O(41^{8}) $ &
$ 5384419950679 + O(41^{8}) $\\
$ 38176784853304 + O(59^{8}) $ & $ 63020796753579 + O(59^{8}) $
& $ 113039802800992 + O(59^{8}) $\\
$ 70602302343232 + O(61^{8}) $ & $ 139895606364222 + O(61^{8}) $
& $ 79733480381568 + O(61^{8}) $\\
$ 433639741922965 + O(71^{8}) $ & $ 576931954734067 + O(71^{8}) $
& $ 8989266238661 + O(71^{8}) $\\
$ 589304115938460 + O(83^{8}) $ & $ 347866087087015 + O(83^{8}) $
& $ 720001059253854 + O(83^{8}) $\\
$ 6769596692483671 + O(97^{8}) $ & $ 4269348271 + O(97^{5}) $ &
$ 3993521258998096 + O(97^{8}) $\\
\hline
 \end{tabular}
\end{center}

 \vspace{.5in}

\subsection{\textbf{$N=85$}} We have the following data:\\

\begin{center}
    \begin{tabular}{| l | l | l |}
    \hline
$p$-adic regulator $\Reg_p(A/\Q)$ & $p$-adic $L$-value & $p$-adic multiplier $\epsilon_p(A)$ \\
\hline
$ 1015073423894 + O(37^{8}) $ & $ 167411116045 + O(37^{8}) $ & $
1002150510104 + O(37^{8}) $\\
$ 6819810980339 + O(41^{8}) $ & $ 7975900636623 + O(41^{8}) $ &
$ 7153783865856 + O(41^{8}) $\\
$ 8962714100713 + O(53^{8}) $ & $ 1069648287223 + O(53^{7}) $ &
$ 44683460285079 + O(53^{8}) $\\
$ 43568329449 + O(61^{6}) $ & $ 3136884016567 + O(61^{7}) $ & $
72019680061615 + O(61^{8}) $\\
$ 119416997911215 + O(73^{8}) $ & $ 683019204724944 + O(73^{8}) $
& $ 32602153132641 + O(73^{8}) $\\
$ 1942338381733272 + O(89^{8}) $ & $ 3482744225118281 + O(89^{8}) $
& $ 1038134293650945 + O(89^{8}) $\\
$ 5147606270477176 + O(97^{8}) $ & $ 2836855197 + O(97^{5}) $ &
$ 3784121167774074 + O(97^{8}) $\\
\hline
 \end{tabular}
\end{center}

 \clearpage
 \subsection{\textbf{$N=93$}} We have the following data:\\

\begin{center}
    \begin{tabular}{| l | l | l |}
    \hline
$p$-adic regulator $\Reg_p(A/\Q)$ & $p$-adic $L$-value & $p$-adic multiplier $\epsilon_p(A)$ \\
\hline
$ 185741420 + O(11^{8}) $ & $ 151839057 + O(11^{8}) $ & $
11051656 + O(11^{8}) $\\
$ 405483221 + O(13^{8}) $ & $ 670790176 + O(13^{8}) $ & $
230057694 + O(13^{8}) $\\
$ 43885519955 + O(23^{8}) $ & $ 26161319539 + O(23^{8}) $ & $
47769373949 + O(23^{8}) $\\
$ 336788503314 + O(29^{8}) $ & $ 484038257980 + O(29^{8}) $ & $
222828561623 + O(29^{8}) $\\
$ 1678741468628 + O(37^{8}) $ & $ 2569674002391 + O(37^{8}) $ &
$ 1378603735422 + O(37^{8}) $\\
$ 5324074002210 + O(43^{8}) $ & $ 8725984878581 + O(43^{8}) $ &
$ 3160248767946 + O(43^{8}) $\\
$ 16824305598488 + O(47^{8}) $ & $ 15669575995471 + O(47^{8}) $
& $ 8286455636222 + O(47^{8}) $\\
$ 4960862919215 + O(53^{8}) $ & $ 49317038818954 + O(53^{8}) $ &
$ 23178143892193 + O(53^{8}) $\\
$ 143070222270789 + O(61^{8}) $ & $ 92506965732666 + O(61^{8}) $
& $ 122923764909639 + O(61^{8}) $\\
$ 279322082363042 + O(67^{8}) $ & $ 74111413499770 + O(67^{8}) $
& $ 180555937022120 + O(67^{8}) $\\
$ 430169136747961 + O(73^{8}) $ & $ 471513912864315 + O(73^{8}) $
& $ 411163460552347 + O(73^{8}) $\\
$ 1384453915387035 + O(79^{8}) $ & $ 1232301086171477 + O(79^{8}) $
& $ 503819572894975 + O(79^{8}) $\\
$ 2228621109604011 + O(83^{8}) $ & $ 1304530016876211 + O(83^{8}) $
& $ 1071794429225898 + O(83^{8}) $\\
$ 1081659147745931 + O(89^{8}) $ & $ 1330994142123689 + O(89^{8}) $
& $ 1518075886594725 + O(89^{8}) $\\
\hline
 \end{tabular}
\end{center}

 \vspace{.5in}

 \subsection{\textbf{$N=103$}} We have the following data:\\

\begin{center}
    \begin{tabular}{| l | l | l |}
    \hline
$p$-adic regulator $\Reg_p(A/\Q)$ & $p$-adic $L$-value & $p$-adic multiplier $\epsilon_p(A)$ \\
\hline
$ 147377758 + O(11^{8}) $ & $ 86486502 + O(11^{8}) $ & $
192773925 + O(11^{8}) $\\
$ 489193484 + O(13^{8}) $ & $ 67428377 + O(13^{8}) $ & $
337691501 + O(13^{8}) $\\
$ 15204606664 + O(19^{8}) $ & $ 10638300382 + O(19^{8}) $ & $
12173049603 + O(19^{8}) $\\
$ 66216995216 + O(23^{8}) $ & $ 18109392006 + O(23^{8}) $ & $
45043095109 + O(23^{8}) $\\
$ 5372718408 + O(29^{8}) $ & $ 76731347688 + O(29^{8}) $ & $
12536647436 + O(29^{8}) $\\
$ 6799091682040 + O(41^{8}) $ & $ 4391281006909 + O(41^{8}) $ &
$ 5303120857798 + O(41^{8}) $\\
$ 23467041445332 + O(47^{8}) $ & $ 5937816898560 + O(47^{8}) $ &
$ 1847891549858 + O(47^{8}) $\\
$ 9449958206985 + O(53^{8}) $ & $ 6585582284426 + O(53^{8}) $ &
$ 47071170371848 + O(53^{8}) $\\
$ 55788681659810 + O(59^{8}) $ & $ 58416917952322 + O(59^{8}) $
& $ 134523844728309 + O(59^{8}) $\\
$ 180708198470076 + O(61^{8}) $ & $ 86273076603078 + O(61^{8}) $
& $ 91320952633362 + O(61^{8}) $\\
$ 304798054862709 + O(71^{8}) $ & $ 23644536785282 + O(71^{8}) $
& $ 318837731560077 + O(71^{8}) $\\
$ 651632900917334 + O(73^{8}) $ & $ 128186925484 + O(73^{6}) $ &
$ 86431680403618 + O(73^{8}) $\\
$ 1422073004111088 + O(79^{8}) $ & $ 162819364440040 + O(79^{8}) $
& $ 1353067258168647 + O(79^{8}) $\\
$ 2204776989584744 + O(83^{8}) $ & $ 290525162365 + O(83^{6}) $
& $ 1489257165084816 + O(83^{8}) $\\
$ 3419478873681093 + O(89^{8}) $ & $ 13722104837501 + O(89^{7}) $
& $ 2011257469143583 + O(89^{8}) $\\
$ 5187359554281130 + O(97^{8}) $ & $ 588713923936 + O(97^{6}) $
& $ 291254315420391 + O(97^{8}) $\\
\hline
 \end{tabular}
\end{center}

 \clearpage
 \subsection{\textbf{$N=107$}} We have the following data:\\

\begin{center}
    \begin{tabular}{| l | l | l |}
    \hline
$p$-adic regulator $\Reg_p(A/\Q)$ & $p$-adic $L$-value & $p$-adic multiplier $\epsilon_p(A)$ \\
\hline
$ 100037184 + O(13^{8}) $ & $ 381034778 + O(13^{8}) $ & $
725504508 + O(13^{8}) $\\
$ 5164824485 + O(17^{8}) $ & $ 2251756830 + O(17^{8}) $ & $
6548185060 + O(17^{8}) $\\
$ 4948122310 + O(19^{8}) $ & $ 410682533 + O(19^{8}) $ & $
11770828305 + O(19^{8}) $\\
$ 2233155353996 + O(37^{8}) $ & $ 782254360600 + O(37^{8}) $ & $
499167048517 + O(37^{8}) $\\
$ 7693933727093 + O(41^{8}) $ & $ 1126806110759 + O(41^{8}) $ &
$ 941047246375 + O(41^{8}) $\\
$ 1985728518871 + O(43^{8}) $ & $ 11350309489829 + O(43^{8}) $ &
$ 1490080422844 + O(43^{8}) $\\
$ 9236325503676 + O(47^{8}) $ & $ 9352977397857 + O(47^{8}) $ &
$ 3877324130340 + O(47^{8}) $\\
$ 66592510503713 + O(59^{8}) $ & $ 126498662012390 + O(59^{8}) $
& $ 66521146158463 + O(59^{8}) $\\
$ 73605475872145 + O(61^{8}) $ & $ 67998854641813 + O(61^{8}) $
& $ 149495311709314 + O(61^{8}) $\\
$ 215631855830774 + O(67^{8}) $ & $ 276136144242399 + O(67^{8}) $
& $ 300204582979356 + O(67^{8}) $\\
$ 235988007934369 + O(71^{8}) $ & $ 4203263257242 + O(71^{7}) $
& $ 297538516047502 + O(71^{8}) $\\
$ 1405009654786451 + O(79^{8}) $ & $ 1425538781612665 + O(79^{8}) $
& $ 637219753066297 + O(79^{8}) $\\
$ 966246807067004 + O(83^{8}) $ & $ 29073683 + O(83^{5}) $ & $
122301722091732 + O(83^{8}) $\\
\hline
 \end{tabular}
\end{center}

 \subsection{\textbf{$N=115$}} We have the following data:\\

\begin{center}
    \begin{tabular}{| l | l | l |}
    \hline
$p$-adic regulator $\Reg_p(A/\Q)$ & $p$-adic $L$-value & $p$-adic multiplier $\epsilon_p(A)$ \\
\hline
$ 151819184 + O(11^{8}) $ & $ 96031694 + O(11^{8}) $ & $
55470083 + O(11^{8}) $\\
$ 6070540659 + O(17^{8}) $ & $ 2602174031 + O(17^{8}) $ & $
2479111430 + O(17^{8}) $\\
$ 443043366998 + O(37^{8}) $ & $ 3207000318071 + O(37^{8}) $ & $
518402902203 + O(37^{8}) $\\
$ 10506890337861 + O(43^{8}) $ & $ 466034248434 + O(43^{8}) $ &
$ 3160248767946 + O(43^{8}) $\\
$ 25938666299194 + O(53^{8}) $ & $ 43404652273198 + O(53^{8}) $
& $ 24704105954182 + O(53^{8}) $\\
$ 68828469915327 + O(59^{8}) $ & $ 15822514736163 + O(59^{8}) $
& $ 112268718282797 + O(59^{8}) $\\
$ 117125015025879 + O(61^{8}) $ & $ 99513360280408 + O(61^{8}) $
& $ 144234417021077 + O(61^{8}) $\\
$ 117261157211649 + O(67^{8}) $ & $ 369365142758789 + O(67^{8}) $
& $ 388573100762289 + O(67^{8}) $\\
$ 38346420175144 + O(79^{8}) $ & $ 1042744621946608 + O(79^{8}) $
& $ 856522414733559 + O(79^{8}) $\\
$ 232154244720909 + O(83^{8}) $ & $ 1238796074898239 + O(83^{8}) $
& $ 2232922964727286 + O(83^{8}) $\\
$ 3680613169329886 + O(89^{8}) $ & $ 1982664616252635 + O(89^{8}) $
& $ 1447736508567520 + O(89^{8}) $\\
$ 337111037730418 + O(97^{8}) $ & $ 5523549952859660 + O(97^{8}) $
& $ 5537610452725212 + O(97^{8}) $\\
\hline
 \end{tabular}
\end{center}

 \subsection{\textbf{$N=125, A$}} We have the following data:\\

\begin{center}
    \begin{tabular}{| l | l | l |}
    \hline
$p$-adic regulator $\Reg_p(A/\Q)$ & $p$-adic $L$-value & $p$-adic multiplier $\epsilon_p(A)$ \\
\hline
$ 298562498 + O(13^{8}) $ & $ 592894408 + O(13^{8}) $ & $
337691501 + O(13^{8}) $\\
$ 6712555657 + O(19^{8}) $ & $ 7153379737 + O(19^{8}) $ & $
7352726322 + O(19^{8}) $\\
$ 28761182485 + O(23^{8}) $ & $ 19244567041 + O(23^{8}) $ & $
47769373949 + O(23^{8}) $\\
$ 1610334394992 + O(37^{8}) $ & $ 2619837199442 + O(37^{8}) $ &
$ 518402902203 + O(37^{8}) $\\
$ 5827125727855 + O(47^{8}) $ & $ 2438975823319 + O(47^{8}) $ &
$ 20137488978024 + O(47^{8}) $\\
$ 43827481404730 + O(53^{8}) $ & $ 12237145144494 + O(53^{8}) $
& $ 45444073360562 + O(53^{8}) $\\
$ 934069839446 + O(59^{8}) $ & $ 138331812786050 + O(59^{8}) $ &
$ 119463235911829 + O(59^{8}) $\\
$ 94940897306587 + O(61^{8}) $ & $ 86820693223899 + O(61^{8}) $
& $ 134545325721836 + O(61^{8}) $\\
$ 344652595573416 + O(67^{8}) $ & $ 285200220171958 + O(67^{8}) $
& $ 55395450190703 + O(67^{8}) $\\
$ 494778091759992 + O(73^{8}) $ & $ 211635131306627 + O(73^{8}) $
& $ 338214791799846 + O(73^{8}) $\\
$ 911058348384486 + O(83^{8}) $ & $ 15062127863580 + O(83^{7}) $
& $ 1695835531921770 + O(83^{8}) $\\
$ 3812663593637783 + O(89^{8}) $ & $ 140584030153 + O(89^{6}) $
& $ 3229221353736449 + O(89^{8}) $\\
$ 7743507247513256 + O(97^{8}) $ & $ 72751227747664 + O(97^{7}) $
& $ 6602401135477806 + O(97^{8}) $\\
\hline
 \end{tabular}
\end{center}

  \subsection{\textbf{$N=133, B$}} We have the following data:\\

\begin{center}
    \begin{tabular}{| l | l | l |}
    \hline
$p$-adic regulator $\Reg_p(A/\Q)$ & $p$-adic $L$-value & $p$-adic multiplier $\epsilon_p(A)$ \\
\hline
$ 4554714851 + O(17^{8}) $ & $ 1400369830 + O(17^{8}) $ & $
1530767973 + O(17^{8}) $\\
$ 482641533 + O(29^{6}) $ & $ 224834369110 + O(29^{8}) $ & $
188246220652 + O(29^{8}) $\\
$ 285247284517 + O(31^{8}) $ & $ 644745508559 + O(31^{8}) $ & $
65426082523 + O(31^{8}) $\\
$ 873461875052 + O(41^{8}) $ & $ 5913841764921 + O(41^{8}) $ & $
5173622706020 + O(41^{8}) $\\
$ 6395433286380 + O(43^{8}) $ & $ 5250591893580 + O(43^{8}) $ &
$ 7173815953060 + O(43^{8}) $\\
$ 40174155934745 + O(53^{8}) $ & $ 9436443664 + O(53^{6}) $ & $
47071170371848 + O(53^{8}) $\\
$ 388303423009987 + O(67^{8}) $ & $ 74275805470 + O(67^{6}) $ &
$ 287829738202699 + O(67^{8}) $\\
$ 582542046575002 + O(73^{8}) $ & $ 47404160292 + O(73^{6}) $ &
$ 214334244118640 + O(73^{8}) $\\
$ 997934987934019 + O(79^{8}) $ & $ 1692929332309 + O(79^{7}) $
& $ 85649658584845 + O(79^{8}) $\\
$ 337083794306147 + O(83^{8}) $ & $ 8815903470 + O(83^{6}) $ & $
1446666792043837 + O(83^{8}) $\\
$ 3826161118964265 + O(89^{8}) $ & $ 689438763 + O(89^{5}) $ & $
1571471061650586 + O(89^{8}) $\\
\hline
 \end{tabular}
\end{center}

  \subsection{\textbf{$N=147$}} We have the following data:\\

\begin{center}
    \begin{tabular}{| l | l | l |}
    \hline
$p$-adic regulator $\Reg_p(A/\Q)$ & $p$-adic $L$-value & $p$-adic multiplier $\epsilon_p(A)$ \\
\hline
$ 434194800 + O(13^{8}) $ & $ 772553365 + O(13^{8}) $ & $
69777210 + O(13^{8}) $\\
$ 3085885399 + O(19^{8}) $ & $ 14351355419 + O(19^{8}) $ & $
10124513344 + O(19^{8}) $\\
$ 57105870 + O(23^{6}) $ & $ 56314647135 + O(23^{8}) $ & $
77688619426 + O(23^{8}) $\\
$ 598807495296 + O(31^{8}) $ & $ 500890389807 + O(31^{8}) $ & $
226083261470 + O(31^{8}) $\\
$ 1255556858069 + O(37^{8}) $ & $ 728534804896 + O(37^{8}) $ & $
925099803678 + O(37^{8}) $\\
$ 3028914438423 + O(43^{8}) $ & $ 5310811645878 + O(43^{8}) $ &
$ 1448245155768 + O(43^{8}) $\\
$ 21722415097178 + O(53^{8}) $ & $ 32574036544128 + O(53^{8}) $
& $ 14098536063957 + O(53^{8}) $\\
$ 9696531871680 + O(61^{8}) $ & $ 147051772023912 + O(61^{8}) $
& $ 127085340697404 + O(61^{8}) $\\
$ 252460432397529 + O(67^{8}) $ & $ 2933215623449 + O(67^{7}) $
& $ 15805729099128 + O(67^{8}) $\\
$ 319985315705867 + O(71^{8}) $ & $ 4744056079140 + O(71^{7}) $
& $ 454718387048106 + O(71^{8}) $\\
$ 696485497462517 + O(73^{8}) $ & $ 75645384726 + O(73^{6}) $ &
$ 263081220212640 + O(73^{8}) $\\
$ 1036811888178773 + O(79^{8}) $ & $ 965207536 + O(79^{5}) $ & $
74020902743243 + O(79^{8}) $\\
$ 4273472549945572 + O(97^{8}) $ & $ 6770845150 + O(97^{5}) $ &
$ 7422648274246094 + O(97^{8}) $\\
\hline
 \end{tabular}
\end{center}

   \subsection{\textbf{$N=161$}} We have the following data:\\

\begin{center}
    \begin{tabular}{| l | l | l |}
    \hline
$p$-adic regulator $\Reg_p(A/\Q)$ & $p$-adic $L$-value & $p$-adic multiplier $\epsilon_p(A)$ \\
\hline
$ 171933135 + O(11^{8}) $ & $ 104178769 + O(11^{8}) $ & $
48803991 + O(11^{8}) $\\
$ 16676191757 + O(19^{8}) $ & $ 8396822512 + O(19^{8}) $ & $
10186228540 + O(19^{8}) $\\
$ 3000539180980 + O(37^{8}) $ & $ 2959738471101 + O(37^{8}) $ &
$ 1378603735422 + O(37^{8}) $\\
$ 4799012913812 + O(43^{8}) $ & $ 820015420 + O(43^{6}) $ & $
7358928540810 + O(43^{8}) $\\
$ 33038825747471 + O(53^{8}) $ & $ 14752347930 + O(53^{6}) $ & $
37415160388754 + O(53^{8}) $\\
$ 144048375212404 + O(59^{8}) $ & $ 67829338510607 + O(59^{8}) $
& $ 113039802800992 + O(59^{8}) $\\
$ 989091293021 + O(61^{8}) $ & $ 62763431617869 + O(61^{8}) $ &
$ 191252449121304 + O(61^{8}) $\\
$ 93625125465306 + O(67^{8}) $ & $ 43269028077 + O(67^{6}) $ & $
228681540167106 + O(67^{8}) $\\
$ 372742847896101 + O(79^{8}) $ & $ 16260012523515 + O(79^{7}) $
& $ 1380060506871347 + O(79^{8}) $\\
$ 133689266642605 + O(83^{8}) $ & $ 27328857470 + O(83^{6}) $ &
$ 1997163923487638 + O(83^{8}) $\\
$ 109602346601919 + O(89^{8}) $ & $ 3775670578 + O(89^{5}) $ & $
404117712562583 + O(89^{8}) $\\
$ 4449889265258731 + O(97^{8}) $ & $ 8056056109 + O(97^{5}) $ &
$ 3796862465389610 + O(97^{8}) $\\
\hline
 \end{tabular}
\end{center}

   \subsection{\textbf{$N=165$}} We have the following data:\\

\begin{center}
    \begin{tabular}{| l | l | l |}
    \hline
$p$-adic regulator $\Reg_p(A/\Q)$ & $p$-adic $L$-value & $p$-adic multiplier $\epsilon_p(A)$ \\
\hline
$ 2478665 + O(7^{9}) $ & $ 988615 + O(7^{8}) $ & $ 2047938 +
O(7^{8}) $\\
$ 7577669996 + O(13^{9}) $ & $ 546478360 + O(13^{8}) $ & $
19120487 + O(13^{8}) $\\
$ 1345832 + O(17^{5}) $ & $ 115518752 + O(17^{8}) $ & $
6743866153 + O(17^{8}) $\\
$ 317314039860 + O(19^{9}) $ & $ 15454527827 + O(19^{8}) $ & $
16701261693 + O(19^{8}) $\\
$ 1197529401836 + O(23^{9}) $ & $ 21430827992 + O(23^{8}) $ & $
32283075894 + O(23^{8}) $\\
$ 182820405709 + O(29^{9}) $ & $ 462793840863 + O(29^{8}) $ & $
167834932128 + O(29^{8}) $\\
$ 30402585606264 + O(37^{9}) $ & $ 145763317789 + O(37^{8}) $ &
$ 1905855970461 + O(37^{8}) $\\
$ 300867423531184 + O(41^{9}) $ & $ 4638175450295 + O(41^{8}) $
& $ 7243944162192 + O(41^{8}) $\\
$ 453491841293220 + O(43^{9}) $ & $ 10502890759714 + O(43^{8}) $
& $ 981850330755 + O(43^{8}) $\\
$ 841500704008115 + O(47^{9}) $ & $ 22177682954670 + O(47^{8}) $
& $ 15840901508219 + O(47^{8}) $\\
$ 108480654690546 + O(53^{9}) $ & $ 812837848921 + O(53^{7}) $ &
$ 17925543534180 + O(53^{8}) $\\
$ 5924946879989069 + O(59^{9}) $ & $ 2021593887077 + O(59^{7}) $
& $ 32899436516884 + O(59^{8}) $\\
$ 11004676059690151 + O(61^{9}) $ & $ 1114772875983 + O(61^{7}) $
& $ 102536167075224 + O(61^{8}) $\\
$ 6402333518135195 + O(67^{9}) $ & $ 1539009404714 + O(67^{7}) $
& $ 15805729099128 + O(67^{8}) $\\
$ 18266464992713450 + O(71^{9}) $ & $ 98457973781 + O(71^{6}) $
& $ 34536869719889 + O(71^{8}) $\\
$ 42183534718264644 + O(73^{9}) $ & $ 6514036760733 + O(73^{7}) $
& $ 39739465931437 + O(73^{8}) $\\
$ 1224455456912234 + O(79^{8}) $ & $ 1812891052 + O(79^{5}) $ &
$ 186442021878008 + O(79^{8}) $\\
$ 124036966428761339 + O(83^{9}) $ & $ 203487283131 + O(83^{6}) $
& $ 1504086404024377 + O(83^{8}) $\\
$ 243696118400513337 + O(89^{9}) $ & $ 2097535192 + O(89^{5}) $
& $ 1826662988317474 + O(89^{8}) $\\
$ 635540819872824429 + O(97^{9}) $ & $ 1063985237 + O(97^{5}) $
& $ 4391966065852909 + O(97^{8}) $\\
\hline
 \end{tabular}
\end{center}

 \vspace{.5in}

  \subsection{\textbf{$N=167$}} We have the following data:\\

\begin{center}
    \begin{tabular}{| l | l | l |}
    \hline
$p$-adic regulator $\Reg_p(A/\Q)$ & $p$-adic $L$-value & $p$-adic multiplier $\epsilon_p(A)$ \\
\hline
$ 19432714 + O(7^{9}) $ & $ 2251 + O(7^{4}) $ & $ 307185 +
O(7^{8}) $\\
$ 13117611 + O(13^{7}) $ & $ 666390377 + O(13^{8}) $ & $
526042526 + O(13^{8}) $\\
$ 1908862518313 + O(17^{10}) $ & $ 2314174880 + O(17^{8}) $ & $
5234654956 + O(17^{8}) $\\
$ 2053675284265 + O(19^{10}) $ & $ 7656154501 + O(19^{8}) $ & $
14340680958 + O(19^{8}) $\\
$ 27719111127295 + O(23^{10}) $ & $ 41736439730 + O(23^{8}) $ &
$ 16647712571 + O(23^{8}) $\\
$ 405640880151858 + O(31^{10}) $ & $ 714765172682 + O(31^{8}) $
& $ 358709025654 + O(31^{8}) $\\
$ 57769565310991429 + O(53^{10}) $ & $ 19704952386 + O(53^{6}) $
& $ 17681237786119 + O(53^{8}) $\\
$ 501328316424338015 + O(59^{10}) $ & $ 64487204069600 + O(59^{8}) $
& $ 140344643451642 + O(59^{8}) $\\
$ 1243828341260907954 + O(71^{10}) $ & $ 56573288611 + O(71^{6}) $
& $ 557648531014830 + O(71^{8}) $\\
$ 2079988387733147685 + O(73^{10}) $ & $ 4104591 + O(73^{4}) $ &
$ 759308640111719 + O(73^{8}) $\\
$ 20365783254113182401 + O(89^{10}) $ & $ 5065696436 + O(89^{5}) $
& $ 1447017073110591 + O(89^{8}) $\\
$ 67990777180272953115 + O(97^{10}) $ & $ 23364634 + O(97^{4}) $
& $ 6376229493766338 + O(97^{8}) $\\
\hline
 \end{tabular}
\end{center}

 \clearpage
   \subsection{\textbf{$N=177$}} We have the following data:\\

\begin{center}
    \begin{tabular}{| l | l | l |}
    \hline
$p$-adic regulator $\Reg_p(A/\Q)$ & $p$-adic $L$-value & $p$-adic multiplier $\epsilon_p(A)$ \\
\hline
$ 1072267 + O(7^{8}) $ & $ 1192 + O(7^{4}) $ & $ 507488 +
O(7^{8}) $\\
$ 9772408 + O(19^{6}) $ & $ 2558009183 + O(19^{8}) $ & $
11268267357 + O(19^{8}) $\\
$ 27690468499 + O(23^{8}) $ & $ 51308343838 + O(23^{8}) $ & $
3788873485 + O(23^{8}) $\\
$ 141718660962 + O(29^{8}) $ & $ 391909937451 + O(29^{8}) $ & $
65127401733 + O(29^{8}) $\\
$ 265288097732 + O(31^{8}) $ & $ 167635394515 + O(31^{8}) $ & $
519021947371 + O(31^{8}) $\\
$ 1019326123826 + O(37^{8}) $ & $ 2370016933013 + O(37^{8}) $ &
$ 1021993916814 + O(37^{8}) $\\
$ 952644023485 + O(41^{8}) $ & $ 5888249521909 + O(41^{8}) $ & $
5919384948361 + O(41^{8}) $\\
$ 21867793727731 + O(47^{8}) $ & $ 3399186192198 + O(47^{8}) $ &
$ 18039627327365 + O(47^{8}) $\\
$ 54813744728211 + O(61^{8}) $ & $ 136070004398022 + O(61^{8}) $
& $ 149495311709314 + O(61^{8}) $\\
$ 374976464608823 + O(73^{8}) $ & $ 92506712920 + O(73^{6}) $ &
$ 469524064138469 + O(73^{8}) $\\
$ 2024750045809193 + O(83^{8}) $ & $ 1537698302 + O(83^{6}) $ &
$ 2034477952337988 + O(83^{8}) $\\
\hline
 \end{tabular}
\end{center}

 \vspace{.5in}

   \subsection{\textbf{$N=188$}} We have the following data:\\

\begin{center}
    \begin{tabular}{| l | l | l |}
    \hline
$p$-adic regulator $\Reg_p(A/\Q)$ & $p$-adic $L$-value & $p$-adic multiplier $\epsilon_p(A)$ \\
\hline
$ 5623044 + O(7^{8}) $ & $ 1259 + O(7^{4}) $ & $ 507488 +
O(7^{8}) $\\
$ 4478725 + O(11^{7}) $ & $ 150222285 + O(11^{8}) $ & $
143254320 + O(11^{8}) $\\
$ 775568547 + O(13^{8}) $ & $ 237088204 + O(13^{8}) $ & $
523887415 + O(13^{8}) $\\
$ 1129909080 + O(17^{8}) $ & $ 6922098082 + O(17^{8}) $ & $
4494443586 + O(17^{8}) $\\
$ 14409374565 + O(19^{8}) $ & $ 15793371104 + O(19^{8}) $ & $
4742010391 + O(19^{8}) $\\
$ 31414366115 + O(23^{8}) $ & $ 210465118 + O(23^{8}) $ & $
45043095109 + O(23^{8}) $\\
$ 2114154456754 + O(37^{8}) $ & $ 1652087821140 + O(37^{8}) $ &
$ 1881820314237 + O(37^{8}) $\\
$ 6279643012659 + O(41^{8}) $ & $ 2066767021277 + O(41^{8}) $ &
$ 4367414685819 + O(41^{8}) $\\
$ 9585122287133 + O(43^{8}) $ & $ 3309737400961 + O(43^{8}) $ &
$ 85925017348 + O(43^{8}) $\\
$ 3328142761956 + O(53^{8}) $ & $ 5143002859 + O(53^{6}) $ & $
6112104707558 + O(53^{8}) $\\
$ 17411023818285 + O(59^{8}) $ & $ 7961878705 + O(59^{6}) $ & $
98405729721193 + O(59^{8}) $\\
$ 102563258757138 + O(61^{8}) $ & $ 216695090848 + O(61^{7}) $ &
$ 137187998566490 + O(61^{8}) $\\
$ 26014679325501 + O(67^{8}) $ & $ 7767410995 + O(67^{6}) $ & $
38320151289262 + O(67^{8}) $\\
$ 490864897182147 + O(71^{8}) $ & $ 16754252742 + O(71^{6}) $ &
$ 530974572239623 + O(71^{8}) $\\
$ 689452389265311 + O(73^{8}) $ & $ 193236387 + O(73^{5}) $ & $
162807895476311 + O(73^{8}) $\\
$ 878760549863821 + O(79^{8}) $ & $ 1745712500 + O(79^{5}) $ & $
1063642669147985 + O(79^{8}) $\\
$ 2070648686579466 + O(83^{8}) $ & $ 2888081539 + O(83^{5}) $ &
$ 1103760059074178 + O(83^{8}) $\\
$ 3431343284115672 + O(89^{8}) $ & $ 1591745960 + O(89^{5}) $ &
$ 1012791564080640 + O(89^{8}) $\\
$ 4259144286293285 + O(97^{8}) $ & $ 21828881 + O(97^{4}) $ & $
6376229493766338 + O(97^{8}) $\\
\hline
 \end{tabular}
\end{center}

\clearpage
   \subsection{\textbf{$N=191$}} We have the following data:\\

\begin{center}
    \begin{tabular}{| l | l | l |}
    \hline
$p$-adic regulator $\Reg_p(A/\Q)$ & $p$-adic $L$-value & $p$-adic multiplier $\epsilon_p(A)$ \\
\hline
$ 4195478 + O(7^{8}) $ & $ 1867 + O(7^{4}) $ & $ 1638463 +
O(7^{8}) $\\
$ 43495803539 + O(23^{8}) $ & $ 62365909362 + O(23^{8}) $ & $
47598354917 + O(23^{8}) $\\
$ 276478270993 + O(31^{8}) $ & $ 411081898951 + O(31^{8}) $ & $
611200443823 + O(31^{8}) $\\
$ 7847912037610 + O(43^{8}) $ & $ 1839263047933 + O(43^{8}) $ &
$ 10085036614653 + O(43^{8}) $\\
$ 3701160666066 + O(47^{8}) $ & $ 16594732090932 + O(47^{8}) $ &
$ 9836262988784 + O(47^{8}) $\\
$ 19837992635361 + O(53^{8}) $ & $ 121641372 + O(53^{5}) $ & $
22289116823061 + O(53^{8}) $\\
$ 207820830309704 + O(71^{8}) $ & $ 80098460638243 + O(71^{8}) $
& $ 318837731560077 + O(71^{8}) $\\
$ 105659818394179 + O(73^{8}) $ & $ 278456920 + O(73^{5}) $ & $
160255667550084 + O(73^{8}) $\\
$ 4330286071100495 + O(97^{8}) $ & $ 12214648 + O(97^{4}) $ & $
1683523428082670 + O(97^{8}) $\\
\hline
 \end{tabular}
\end{center}
 $\qquad$\\

\section{Evidence for a twist of rank 4}\label{sec:r4twist}
    In this section, we present evidence for Conjecture~\ref{pbsd} on a rank 4 twist of a
rank 0 modular abelian surface for the primes~29, 61 and~79.
    Let $X=X_0(31)$.
    According to \cite{galbraith}, an affine equation for $X$ is given by
   \[
        y^2 = (x^3-2x^2-x+3)\cdot (x^3-6 x^2-5 x-1).
    \]
    The Jacobian $A$ of $X$ has rank zero over $\Q$.

    We search for quadratic twists of high rank by searching among quadratic twists
    $\psi$ of small conductor $D$ for some $A_\psi$ whose complex $L$-series seems to
    vanish at $s=1$ up to order at least~4.
    This is the case for $D=-47$.
    We then use a 2-descent on $A_{\psi}$ as implemented in {\tt Magma} to find that the
    rank is at most~4.
    Searching for $\Q$-rational points on $A_{\psi}$ of small height quickly reveals
    subgroups of rank~4, such as the groups $G_p$ described below, thus proving that the
rank $A_{\psi}$ over $\Q$ is indeed~4.

    Using \cite[Theorem~1.17]{bosch-liu:rat_pts} we find that the potentially nontrivial Tamagawa numbers are
    $c_{31}(A_\psi)=1$ and $c_{47}(A_\psi)=16$.
    Moreover, the torsion subgroup is trivial as is the 2-torsion of
    $\Sha(A_{\psi}/\Q)$.
Since the divisors supported in $\Q$-rational points of $X_{\psi}$ do not
    generate a subgroup of finite index, we compute the $p$-adic regulator for $p\in
    \{29,61,79\}$ using
    the second method outlined in Section~\ref{sec:comp_reg}.
    Namely, for each $p$ we find a finite index subgroup $G_p$ of
      $A_{\psi}(\Q)/A_{\psi}(\Q)_{\textrm{tors}}$ generated by the classes of divisors $D_1,\ldots,
      D_4$ such that each $D_i$ has pointwise $\Q_p$-rational support.
      It is then enough to compute the $p$-adic regulator of $G_p$ to find the $p$-adic
      regulator of $A_{\psi}(\Q)/A_{\psi}(\Q)_{\textrm{tors}}$.

      The generators we used are given in Table~\ref{table:gens_indices} in Mumford
      representation, along with the index of $G_p$.
    \begin{table}[h!]
      \begin{center}
        \begin{tabularx}{\textwidth}{ |r|l|X| }
      \hline
      $p$ & index & generators of $G_p$\\
      \hline
      29 & 2 & $
            [x^2 - 7/2x + 49/16,
            581/16x - 1305/32],
            [x^2 - 2x - 1/2,
          -47/2x],$

    $ [x^2 - 5/3x + 5/6,
            47/18x - 517/18],
            [x^2 - 19/3x - 11/4,
            -517/36x - 47/12]
    $\\
  61 & 4 &$[
        x^2 - 7/2x + 49/16,
        581/16x - 1305/32
    ],
    [
        x^2 - 5x + 11/2,
        -235/2x + 423/2
    ],$

$ [
        x^2 - 5x - 7/3,
        235/3x + 94/3
    ],
    [
        x^2 - 5/3x + 5/6,
        -47/18x + 517/18
    ]
 $ \\
  79 & 4 &$[
        x^2 - 7/2x + 49/16,
        581/16x - 1305/32
    ],
    [
        x^2 - x - 1/3,
        -47/3x
    ],$

    $[
        x^2 - 3x - 5/3,
        149/3x + 82/3
    ],
    [
        x^2 - 19/7x + 3/7,
        -1363/49x - 564/49
    ]
 $ \\
  \hline
      \end{tabularx}
    \end{center}
    \caption{$\;\;$Indices and generators, $N=31$ twisted by $D =
-47$}\label{table:gens_indices}
  \end{table}

In order to compute the special values of the $p$-adic $L$-series, we need
to find out the correct normalization factor $\delta^-$ of the minus modular symbol map
associated to $A$.
The twist $A_\chi$ of $A$ by the quadratic character $\chi$ associated to $\Q(\sqrt{-19})$ has rank~0 over $\Q$.
Comparing $[0]^+_{A_\psi}=-4$ to
$    \frac{L(A_\chi,1)}{\Omega^+_{A_\chi}}=1,$
we find that $\delta^-=-\frac{1}{4}$.

For the latter computation, we used that for a minimal equation of $X_\chi$, a basis of the integral
1-forms on $A_\chi$ is given by $(\omega_1,\omega_2)$ as in Remark~\ref{rk:find_eta}.
Since the corresponding fact also holds for $A$, we find that $\eta_\chi=1$ using
\cite{liu:conducteur}.

The data presented in this section proves:
\begin{proposition}\label{prop:r4twist}
Assume that  $\Sha(A_\psi/\Q)$ consists entirely of 2-torsion,
and that the conjectural order of $\Sha(A_{\psi}/\Q)$ is $1$ (numerically it is $1.0000000\ldots$ to as many digits
as we care to compute).
Then Conjecture~\ref{pbsd}
is satisfied up to the respective precision specified in
the table below
for the twist $A_\psi$ of $J_0(31)$ of rank~4
for the primes~29,~61 and~79.
\end{proposition}

The special values of the $p$-adic $L$-series, the $p$-adic regulators and the
$p$-adic multipliers for $p\in\{29,61,79\}$ are given in
 the following table:\\

\begin{center}
\begin{tabular}{| l | l | l |}
    \hline
$p$-adic regulator $\Reg_p(A_\psi/\Q)$ & $p$-adic $L$-value & $p$-adic multiplier
    $\epsilon_p(A_\psi)$ \\
\hline
$351486231941615978+O(29^{12})$  & $202402009906+ O(29^8) $ &$ 423952915488+O(29^8)$\\
$1650697608489237057465+O(61^{12})$  & $4326648666405+ O(61^8) $ &$10267186717780 +O(61^8)$\\
    $8155329946924028539010+O(79^{12})$  & $1513185184992411+ O(79^8) $ &$1431106352547896
+O(79^8)$\\
\hline
 \end{tabular}
\end{center}
 $\qquad$\\
 \begin{remark}
Assume that  $\Sha(A_\chi/\Q)$ consists entirely of 2-torsion.
Then we also verified the classical (and hence for good ordinary primes $p$ the $p$-adic)
conjecture of Birch and
Swinnerton-Dyer for the rank~0 twist $A_\chi$,
since all Tamagawa numbers, the order of the torsion subgroup and the order of
$\Sha(A_\chi/\Q)[2]$ are easily seen to be equal to~1.
\end{remark}

\bibliography{biblio}
\bibliographystyle{plain}
\end{document}